\documentclass[10pt]{article}
\usepackage{amsthm,fullpage,tikz,xcolor}
\usepackage{graphics,verbatim,hyperref}
\usetikzlibrary{patterns}
\usepackage{wrapfig}
\usepackage{caption}
\newtheorem{lem}{Lemma}
\newtheorem{clm}{Claim}
\newtheorem{thm}{Theorem}

\newtheorem*{mainthm}{Main Theorem}
\newtheorem*{maincor}{Main Corollary}
\newtheorem*{thmcor}{Theorem~1}

\newtheorem{conj}{Conjecture}
\theoremstyle{definition}
\newtheorem*{remark}{Remark}

\def\vph{\varphi}

\RequirePackage{marginnote,hyperref}
\newcommand{\aside}[1]{\marginnote{\scriptsize{#1}}[0cm]}
\newcommand{\aaside}[2]{\marginnote{\scriptsize{#1}}[#2]}
\newcommand\Emph[1]{\emph{#1}\aside{#1}}
\newcommand\EmphE[2]{\emph{#1}\aaside{#1}{#2}}

\usepackage{color}
\usepackage{amssymb}
\usepackage{amsmath}
\usepackage{graphicx}

\definecolor{mypurple}{RGB}{208,134,255}
\definecolor{myblue}{RGB}{10,120,253}
\definecolor{myredred}{RGB}{163,0,0}
\colorlet{myred}{myredred!70}
\definecolor{mygreen}{RGB}{126,198,52}
\definecolor{myorange}{RGB}{244,154,33}

\title{In Most 6-regular Toroidal Graphs \\ All 5-colorings are Kempe Equivalent}
\author{Daniel W. Cranston\thanks{%
Department of Computer Science, Virginia Commonwealth
University, Richmond, VA, USA;
\texttt{dcranston@vcu.edu}
} 
\and Reem Mahmoud\thanks{%
Department of Mathematics and Applied Mathematics, Virginia Commonwealth
University, Richmond, VA, USA; 
\texttt{mahmoudr@vcu.edu}
}
}

\begin{document}
\maketitle

\begin{abstract}
A Kempe swap in a proper coloring interchanges the colors on some maximal
connected 2-colored subgraph.  Two $k$-colorings are $k$-equivalent if we can
transform one into the other using Kempe swaps.  The triangulated toroidal grid,
$T[m\times n]$, is formed from (a toroidal embedding of) the Cartesian product
of $C_m$ and $C_n$ by adding parallel diagonals inside all 4-faces.  
Mohar and Salas showed that not all 4-colorings of $T[m\times n]$ are 4-equivalent.
In contrast, Bonamy, Bousquet, Feghali, and Johnson showed that all 6-colorings of
$T[m\times n]$ are 6-equivalent.  They asked whether the same is true for 5-colorings.
We answer their question affirmatively when $m,n\ge 6$.
Further, we show that if $G$ is 6-regular with a toroidal embedding where every
non-contractible cycle has length at least 7, then all 5-colorings of $G$ are 5-equivalent.  
Our results relate to the antiferromagnetic Pott's model in
statistical mechanics.
\end{abstract}

\section{Introduction}
Given a $k$-coloring $\vph$ of a graph $G$ and colors $\alpha, \beta
\in\{1,\ldots,k\}$, a \Emph{Kempe swap} recolors a component of the subgraph
induced by colors $\alpha$ and $\beta$, interchanging those colors on that component,  
which is called a \EmphE{Kempe component}{-2mm}.
Two $k$-colorings, $\vph_1$ and $\vph_2$ of a graph $G$ are (Kempe)
\EmphE{$k$-equivalent}{4mm}, if there exists a sequence
of $k$-colorings of $G$, beginning with $\vph_1$ and ending with $\vph_2$, such
that each two successive colorings 
differ by only a single Kempe swap.  
Problems regarding the $k$-equivalence of $k$-colorings often go by the name
\emph{reconfiguration}.  We form an auxilliary graph $H$ which has as its
vertices all $k$-colorings of $G$, and two vertices of $H$ are adjacent if the
corresponding colorings differ by only a single Kempe swap. Previous work has
focused on deciding whether $H$ is connected~\cite{mohar,FJP,BBFJ,CvdHJ,CvdHJ2} 
and (when $H$ is connected) on
determining or bounding the diameter of $H$~\cite{BJLPP, BP, feghali-paths, feghali-JCTB}.

The \Emph{$k$-state Pott's model} is among those most widely studied in
statistical mechanics, offering insight ino many areas of solid-state physics.
The model assigns to each vertex  $v$ of an underlying graph $G$ a spin
$\sigma(v)\in\{1,\ldots,k\}$.  So each spin of $G$ corresponds to a $k$-coloring
of $G$, not necessarily proper.
In the \emph{ferromagnetic model}\aside{(anti)ferro-magnetic model}, the energy of a spin is low (meaning the spin is
more likely) when most edges have endpoints  with the same color.  In the
\emph{antiferromagnetic model}, which we are concerned with, the energy of a spin is
low when most edges have endpoints with distinct colors.  When the temperature
is 0, every edge has endpoints with distinct colors, so the possible spins of
$G$ are precisely its proper $k$-colorings.

\begin{figure}[!t]
\centering
\begin{tikzpicture}[scale=1.1]
\clip(0,0) rectangle (8.0,5.75);  
\begin{scope}
\clip(0,0) rectangle (8.0,5.5);  
\tikzstyle{uStyle}=[shape = circle, minimum size = 8pt, inner sep =
2.5pt, outer sep = 0pt, draw, fill=white]

\foreach \i in {1,...,5}
\draw[thick] (.5,\i) -- (7.5,\i);

\foreach \i in {1,...,7}
\draw[thick] (\i,.5) -- (\i,5.5);

\foreach \b in {.5, 1.5, ..., 5.5}
\draw[thick] (.5,\b) -- (6-\b, 5.5);

\foreach \b in {.5, 1.5, ..., 7.5}
\draw[thick] (\b,.5) -- (7.5, 8-\b);

\foreach \x in {1, ..., 7}
{
\foreach \y in {1, ..., 5}
{
\draw (\x,\y) node[uStyle] {};
}
}

\draw (.65,5.25) node[draw=none] {\scriptsize{(1,1)}};
\draw (7.35,.75) node[draw=none] {\scriptsize{$(5,7)$}};

\end{scope}
\end{tikzpicture}
\caption{A triangulated toroidal grid $T[5\times 7]$.%
\label{tri-toroidal-fig}}
\end{figure}
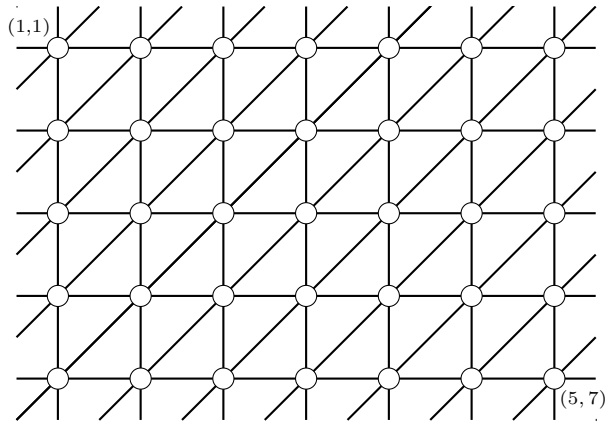

A popular way to simulate the evolution of a graph's spin over time is using the
Wang--Swendsen--Koteck\'{y} (WSK) non-local cluster dynamics.  This is a Markov
Chain Monte Carlo algorithm for the antiferromagnetic Pott's model.  For the WSK
algorithm to function properly, we need that each spin of the graph has positive
probability of eventually transforming to each other spin, i.e., the Markov
chain must be ergodic (connected).  Two states in the Markov chain will be
adjacent exactly when they differ by a single Kempe swap.  Thus, in the language
of graph coloring, ergodicity requires that every two proper $k$-colorings of
$G$ are $k$-equivalent.  This is not true in general, so physicists typically
study highly structured graphs, often with a grid-like structure on the plane
or torus.  (More details of the Pott's model are given in~\cite{MS}.)

The $a\times b$
\EmphE{toroidal grid}{0mm} is the cartesian product of cycles of lengths $a$ and $b$.
The $a\times b$ \EmphE{triangulated toroidal grid}{3mm}, $T[a\times b]$, is
formed from the toroidal grid by adding a diagonal inside each 4-face, so that
all diagonals are parallel; see Figure~\ref{tri-toroidal-fig}.  Clearly each
$T[a\times b]$ is a 6-regular toroidal graph.  And when $a\ge 3$ and $b\ge 3$,
the graph $T[a\times b]$ is 4-colorable (see Lemma~\ref{6reg-4col}).  
Mohar and Salas~\cite{MS} showed that not all 4-colorings are 4-equivalent (when
$a$ and $b$ are divisible by 3).  In contrast, Bonamy, Bousquet, Feghali, and
Johnson~\cite{BBFJ} showed that all 6-colorings of $T[a\times b]$ are
6-equivalent. Further, they asked whether all 5-colorings of $T[a\times b]$ are
5-equivalent.  This question motivates the present paper.
We answer the question affirmatively when $a\ge
6$ and $b\ge 6$.  

\begin{thmcor}
\label{main-cor}
If $G$ is a triangulated toroidal grid $T[a\times b]$ with $a\ge 6$ and $b\ge
6$, then all 5-colorings of $G$ are 5-equivalent.
\end{thmcor}

Our proof holds in more generality, but
to state our main result we need one more definition.
For an embedding of a graph $G$ in the torus, the \emph{edge-width} is
the length of the shortest non-contractible cycle.  Negami showed that if a
toroidal graph is 6-regular, then it has a unique embedding\footnote{We omit a
formal definition of unique embedding; informally, it means that we can
transform any embedding to any other by ``sliding'' the graph around the torus,
keeping it embedded throughout this sliding process.} in the torus.  So, for
each 6-regular toroidal graph $G$, by the \Emph{edge-width of $G$}
we mean the edge-width of the unique toroidal embedding of $G$.
The purpose of this paper is to prove the following.  

\begin{mainthm}
If $G$ is a 6-regular toroidal graph with edge-width at least 7, then all
5-colorings of $G$ are 5-equivalent.
\end{mainthm}

Our Main Theorem implies the following corollary, the proof of which can
also be found in Section~\ref{outline-sec}.

\begin{maincor}
If $G$ is a 6-regular toroidal graph on $n$ vertices chosen uniformly at
random, then asymptotically almost surely all 5-colorings of $G$ are
5-equivalent.
\end{maincor}

Kempe swaps were introduced in a failed attempt to prove the 4 Color
Theorem.  But the idea was salvaged by Heawood, who used them to prove
the 5 Color Theorem.  They also play a central role in proving most
results on edge coloring, since in a line graph each Kempe component is
either a cycle~or~a~path. 

We now mention another direction of research where the present work is
relevant.
A graph is \Emph{$d$-degenerate} if each of its subgraphs has minimum degree at
most $d$.
Las Vergnas and Meyniel proved~\cite{LvM} (see Lemma~\ref{easy-lem}) that if $G$
is $d$-degenerate, then all $k$-colorings of $G$ are $k$-equivalent, whenever
$k>d$.  Thus, every planar graph $G$ has all $k$-colorings equivalent whenever
$k>5$.  Meyniel~\cite{meyniel} extended this result to the case $k=5$.  In contrast,
Mohar~\cite{mohar} constructed planar graphs with arbitrarily many
4-colorings no two of which are 4-equivalent.\footnote{For example, begin with 
the cartesian product of $K_3$ and $K_2$, drawn in the plane, and add a 
4-vertex inside each 4-face.  The
resulting graph is 4-chromatic, but has two Kempe non-4-eqivalent 4-colorings.
By gluing $t$ copies of this graph along copies of $K_3$, we form a graph with
$2^t$ 4-equivalence classes.  By gluing carefully, we ensure that the
resulting graph is planar.}
However, he showed that if $G$ is
planar and $\chi(G)=3$, then all 4-colorings of $G$ are 4-equivalent.  
It is easy to check that if $G$ is bipartite, then all 3-colorings of $G$ are
3-equivalent.  By combining these results, Mohar showed that if $G$ is planar,
then all $(\chi(G)+1)$-colorings of $G$ are $(\chi(G)+1)$-equivalent.  
It is natural to ask whether the same result holds for all toroidal graphs.
Our results can be viewed as a step toward answering this question
affirmatively.

We prove the Main Theorem in Sections~\ref{outline-sec} and~\ref{templates-sec}.
Since the proof of Theorem~\ref{main-cor} is very similar to that of the Main
Theorem, we only discuss it in Section~\ref{ext-sec}, where we
sketch how to adapt the proof of the Main Theorem.  It is worth noting that the
Main Theorem immediately implies the case where $a\ge 7$ and $b\ge
7$.  By symmetry, we can assume that $a\le b$.  So the discussion in
Section~\ref{ext-sec} just handles the case when $a=6$.

\section{Proof Outline and Preliminaries}
\label{outline-sec}
\subsection{An Introduction to Good Templates}
Given a coloring $\vph$ of $G$, our idea is to identify the vertices in one
or more independent sets, each of which receives a common color under $\vph$.
If the resulting graph is 4-degenerate, then all of its 5-colorings are
5-equivalent, as shown in Lemma~\ref{easy-lem}; and these 5-colorings
correspond to some of the 5-colorings
of $G$ (precisely those 5-colorings of $G$ where the vertices in each identified
independent set receive a common color).  
A \EmphE{good\\ 4-template}{-4mm} in $G$ is an independent set $T$
of size 4 such that identifying all vertices of $T$ yields a 4-degenerate graph;
see Figure~\ref{4template-ex-fig}.
We show that if $\vph_1$ and $\vph_2$ are 5-colorings that each use a common
color on some good 4-template, say $T_1$ and $T_2$, then $\vph_1$ and $\vph_2$
are 5-equivalent.  We also show that every 5-coloring is 5-equivalent to a
5-coloring that uses a common color on the vertices of some good 4-template.
Together, these two steps prove our Main Theorem.
To formalize this approach, we introduce more terminology. 

A \EmphE{template}{-4mm} $T$ in $G$ is a collection of disjoint independent sets; each
set in $T$ is a \Emph{color} of $T$.  A template with a single color is
\emph{monochromatic}.
A template $T$ \Emph{appears} in a
coloring $\vph$ of $G$ if the vertices in each color of $T$ receive a common
color under $\vph$.  If $T$ appears in $\vph$, then also $\vph$
\Emph{contains} $T$.
By \emph{contracting a template}\aaside{contracting\\ a template}{4mm} $T$, we
mean identifying the
vertices in each color of $T$.  When we contract template $T$ in $G$, the
resulting graph is denoted $G_T$.  A template $T$ is \EmphE{good}{3mm} for $G$ if
$G_T$ is 4-degenerate.  
We will see that every 6-regular toroidal graph $G$ is vertex transitive.  So if
$G$ has a single good template, then it has many of them.  Thus, given a
5-coloring $\vph$, our focus will be on finding a good 4-template contained in
$\vph$ (or some 5-coloring that is Kempe 5-equivalent to $\vph$).
Good templates play a central role in our proof of the Main Theorem.  This is
due to the following three easy lemmas.  The first was originally proved
in~\cite{LvM}; but for completeness, we include the proof.
The third, Lemma~\ref{switch-lem}, holds in more generality, which we discuss in
Section~\ref{ext-sec}.

\begin{lem}
Let $G$ be a graph with a vertex $w$ such that $d(w)<k$.  
If all $k$-colorings of $G-w$ are $k$-equivalent, then also all $k$-colorings of
$G$ are $k$-equivalent.  Thus, if $G$ is $d$-degenerate and $d<k$, then all
$k$-colorings of $G$ are $k$-equivalent.
\label{easy-lem}
\end{lem}
\begin{proof}
The second statement follows from the first by induction on $|V(G)|$, where the
base case $|V(G)|=1$ holds trivially.  Now we prove the first.
Let $w$ be a vertex with $d(w)<k$. 
Let $G':=G-w$.  Let $\vph_1$ and $\vph_2$ denote $k$-colorings of $G$ and let
$\vph_1'$ and $\vph_2'$ denote their restrictions to $G'$.
By hypothesis, all $k$-colorings of $G'$ are $k$-equivalent.
So there exists a sequence $\psi'_1,\ldots,\psi'_s$ of $k$-colorings of $G'$
such that $\psi'_1=\vph_1'$, $\psi'_s=\vph_2'$ and each $\psi'_{i+1}$ differs
from $\psi'_i$ by only a single Kempe swap.  We extend each $\psi'_i$ to a
$k$-coloring $\psi_i$ of $G$ as follows.  Let $\psi_1=\vph_1$.  Suppose that
$\psi'_{i+1}$ differs from $\psi'_i$ by an $\alpha/\beta$ swap at a vertex
$v_i$.  To construct $\psi_{i+1}$ from $\psi_i$ we use the same $\alpha/\beta$
swap at $v_i$, unless $w$ lies in the same $\alpha/\beta$ component as $v_i$ and
has at least 2 neighbors in that component.  In that case, some color 
$\gamma\in \{1,\ldots,k\}$ is not used on the closed neighborhood of $w$.  Now
we first recolor $w$ with $\gamma$, and then use the $\alpha/\beta$ swap at $v_i$.
We call the resulting coloring $\psi_{i+1}$.  By induction on $i$, each
$\psi_i$ restricts to $\psi'_i$ on $G'$.  Thus, $\psi_s$ agrees with $\vph_2$ on
all vertices except for possibly $w$.  If needed, recolor $w$ with its color in
$\vph_2$.  
\end{proof}

\begin{lem}
If $T$ is a good template in a graph $G$, then all 5-colorings of $G$ containing
$T$ are 5-equivalent.
\label{template-lem}
\end{lem}
\begin{proof}
Form $G_T$ from $G$ by contracting $T$.  Note that each 5-coloring $\vph$ of
$G$ containing $T$ corresponds to a 5-coloring $\vph_T$ of $G_T$ (formed by
contracting $T$ in $\vph$).  Since $T$ is good, $G_T$ is 4-degenerate.  By
Lemma~\ref{easy-lem}, all 5-colorings of $G_T$ are 5-equivalent.  If $\eta$ and
$\zeta$ are 5-colorings of $G$ containing $T$, then $\eta_T$ and $\zeta_T$
are 5-equivalent colorings of $G_T$.  Further, this is witnessed by a sequence of
Kempe swaps.  This same sequence of Kempe swaps witnesses the 5-equivalence of
$\eta$ and $\zeta$.  To simulate in $G$ an $\alpha$/$\beta$ swap at a vertex $v$ in
$G_T$, we simply perform an $\alpha$/$\beta$ swap at each vertex in $G$ that
was identified to form $v$.  (If $v\in V(G)$, then this is a single swap; but
if $v$ represents some non-singleton color of $T$, then this may be multiple swaps.)
\end{proof}

\begin{lem}
\label{switch-lem}
Let $G$ be a 4-colorable graph.  If $\vph_1$ and $\vph_2$ are 5-colorings of $G$
that contain monochromatic good templates $T_1$ and $T_2$, then $\vph_1$ and
$\vph_2$ are 5-equivalent.
\end{lem}

\begin{proof}
Let $\vph_0$ be a 4-coloring of $G$; for concreteness, assume that $\vph_0$ does
not use green.  Form $\vph_1'$ and $\vph_2'$ from $\vph_0$ by recoloring the
vertices of $T_1$ and $T_2$, respectively, with green.  Now $\vph_1$ and
$\vph_1'$ are 5-equivalent, by Lemma~\ref{template-lem}, since they both contain
the good template $T_1$.  Similarly, $\vph_2$ and $\vph_2'$ are 5-equivalent.
Finally, $\vph_1'$ and $\vph_2'$ are both 5-equivalent to $\vph_0$, since each
is formed from $\vph_0$ by recoloring an independent set with green (and each
such recoloring step is a valid Kempe swap).
\end{proof}

We will see that all 6-regular toroidal graphs with edge-width at least 7 are
4-colorable (with one exception, which we handle separately).  This was proved
by Yeh and Zhu~\cite{YZ}, building on work of Collins and Hutchinson~\cite{CH}.
 The latter also proved that $T[a\times b]$ is 4-colorable whenever $a\ge 6$
and $b\ge 6$.  So most of the work in proving our Main Theorem (as well as
Theorem~\ref{main-cor}) goes to showing that if $G$ is a 6-regular toroidal
graph with edge-width at least 7, then every 5-coloring of $G$ is 5-equivalent
to a coloring that contains a monochromatic good template.
That is the content of Section~\ref{templates-sec}.

In view of Lemmas~\ref{template-lem} and \ref{switch-lem}, and ideas
in the previous paragraph, we need a tool to prove that certain templates are
good.  
A \emph{4-degeneracy order}\aaside{4-degener-acy~order}{-4mm} of a graph $G$ is an order $\sigma$ of $V(G)$ such
that each vertex in $\sigma$ has at least $d(v)-4$ neighbors that appear earlier
in $\sigma$.
A \emph{4-degeneracy prefix}\aaside{4-degener-acy~prefix}{-1mm} of a graph $G$ is an order
$\sigma$ of some subset of $V(G)$ such that each vertex in $\sigma$ has at least
$d(v)-4$ neighbors that appear earlier in $\sigma$.
A subgraph $H$ of $G$ is \EmphE{locally connected}{-2mm} if each pair of vertices in
$H$ that is at distance two in $G$ is also at distance two in $H$.  A subgraph
$H$ of $G$ is \EmphE{well-behaved}{0mm} if (i) it is locally connected and (ii)
$G-H$ is connected.

\begin{lem}
Let $H$ be a well-behaved subgraph of $G$ and let $V_H$ denote its vertex set.
Let $T$ be a template with all vertices in $V_H$; denote the new vertices in
$G_T$ by $V_T$.  Suppose there exist $v_1,v_2\in V_H\setminus V_T$ and there
exist $v_3,v_4\in V(G_T)\setminus V_H$ such that $v_3$ is a common neighbor of
$v_1,v_2$ and that $v_4$ is a common neighbor of $v_3,v_i$ for some $i\in\{1,2\}$.  
If there exists a 4-degeneracy prefix $\sigma$ for $V_H\setminus V_T$,
then $G_T$ has a 4-degeneracy order, with all vertices of $V_T$ coming last in
the degeneracy order.
\label{good-template-lem}
\end{lem} 
\begin{proof}
We will extend $\sigma$ to all of $V(G_T)\setminus V_H$.  Let $R:=\emptyset$.  We
iteratively add vertices of $V(G_T)\setminus V_H$ to $R$, so that at each step
(i) $G[R]$ is connected and (ii) 
$\sigma$ can be extended to $(V_H\setminus V_T)\cup R$.
Since $v_1,v_2\in V_H\setminus V_T$ and $v_3$ is their common neighbor, we can
add $v_3$ to $R$ and append it to $\sigma$; similarly, we can add $v_4$.
Let $S:=V(G_T)\setminus (V_H\cup R)$.  Now suppose that $S \ne\emptyset$. 
Choose $v\in R$ such that $v$ has a neighbor in $S$.  (This is possible since
$G-H$ is connected.) Denote the neighbors of $v$ in cyclic order by
$w_1,\ldots,w_6$.  Since $G[R]$ is connected, and $|R|\ge 2$, we can assume
that $w_1\in R$ and $w_6\notin R$.  If $w_6\in S$, then we add $w_6$ to $R$
and append it to the prefix, since its neighbors $v$ and $w_1$ are
already in $R$.  So assume that $w_6\in V_H$.  Now $w_2\notin V_H$, since
$w_1,v\notin V_H$ and $H$ is locally connected.
If $w_2\in S$, then we add $w_2$ to
$R$, since its neighbors $w_1$ and $v$ are in $R$.  So assume that $w_2\in R$.
Note that $w_3\notin V_H$, since $H$ is locally connected and $w_6\in V_H$ but
$v\notin V_H$.
Again, if $w_3\in S$, then we add $w_3$ to $R$; so assume that $w_3\in R$.  If
$w_4\in S$, then we add $w_4$ to $R$, so assume that $w_4\notin S$.  If $w_4\in
V_H$, then also $w_5\in V_H$, since $H$ is locally connected and $w_4,w_6\in V_H$.
This contradicts our choice of $v$ as having a neighbor in $S$.
So $w_4\notin V_H$, which implies that $w_4\in R$.  Finally, $w_5\in S$ since
$v$ has some neighbor in $S$, by our choice of $v$.  Thus, we can always grow $R$
and $\sigma$, as desired.
\end{proof}

\begin{figure}
\centering
\begin{tikzpicture}

\begin{scope}
\clip(-1.73,-1.02) rectangle (3.46,2.02);  
\foreach \ang in {0, 60, 120}  
{
\begin{scope}[rotate=\ang]
\foreach \i in {-5,...,5} 
\draw[thick] (-5.5, \i) -- (6,\i);
\end{scope}
}
\end{scope}
\tikzstyle{uStyle}=[shape = circle, minimum size = 15pt, inner sep =
1.5pt, outer sep = 0pt, draw, fill=white]

\begin{scope}[xscale=.577] 
\foreach \i/\j/\col in {
-2/0/orange, 1/1/orange, 1/-1/orange, 4/2/orange}
\draw[very thick] (\i,\j) node[uStyle, draw=my\col, fill=my\col!60] {};  

\foreach \i/\j/\lab in {
0/0/1, 2/0/2, 3/1/3}
\draw[very thick] (\i,\j) node[uStyle, draw=black!60!white, fill=white] {\small{\lab}};
\end{scope}

\begin{scope}[xshift = -2.5in]
\begin{scope}
\clip(-1.73,-1.02) rectangle (3.46,2.02);  
\foreach \ang in {0, 60, 120}  
{
\begin{scope}[rotate=\ang]
\foreach \i in {-5,...,5} 
\draw[thick] (-5.5, \i) -- (6,\i);
\end{scope}
}
\end{scope}
\tikzstyle{uStyle}=[shape = circle, minimum size = 15pt, inner sep =
1.5pt, outer sep = 0pt, draw, fill=white]

\begin{scope}[xscale=.577] 
\foreach \i/\j/\col in {
-2/0/orange, 1/1/orange, 1/-1/orange, 5/1/orange}
\draw[very thick] (\i,\j) node[uStyle, draw=my\col, fill=my\col!60] {};  

\foreach \i/\j/\lab in {
0/0/1, 2/0/2, 3/1/3}
\draw[very thick] (\i,\j) node[uStyle, draw=black!60!white, fill=white] {\small{\lab}};
\end{scope}
\end{scope}

\end{tikzpicture}
\caption{Two examples of good 4-templates in 6-regular toroidal graphs; each
includes a triple centered at 1.  In both examples, 2 and 3 serve as $v_1$
and $v_2$ in Lemma~\ref{good-template-lem}.
It is easy to check that the subgraph induced by the four orange vertices and
the three numbered vertices is locally connected; thus, it is well-behaved.
\label{4template-ex-fig}}
\end{figure}
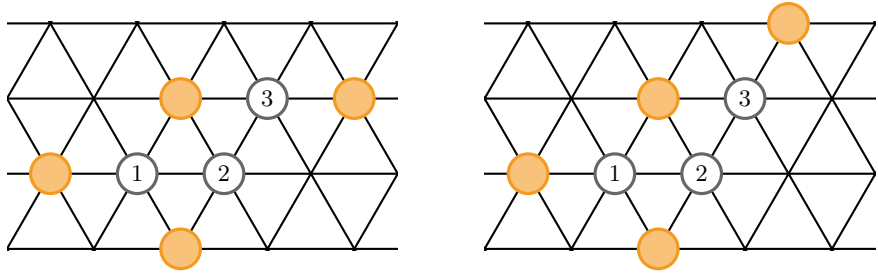

To apply Lemma~\ref{good-template-lem}, we will want to verify that some
subgraphs $H$ are locally connected.  This is fairly easy, but is complicated
slightly by the presence of non-contractible cycles.  So the following lemma is
useful.

\begin{lem}
Let $G$ be a 6-regular toroidal graph with edge-width at least 7.  If $H$ is a
subgraph of $G$ with diameter at most 4, then $H$ is locally connected unless
there exists a vertex $w$ such that $H$ contains at least four neighbors of $w$
but excludes $w$.
\label{loc-help-lem}
\end{lem}
First, suppose instead that there exists a vertex $w$ such that $H$ contains at least
four neighbors of $w$, but excludes $w$.  By Pigeonhole, $H$ contains some pair
of non-adjacent vertices such that $w$ is their only common neighbor.  So, $H$
is \emph{not} locally connected.  Thus, the hypothesis on $w$ in
Lemma~\ref{loc-help-lem} is necessary.

\begin{proof}
Suppose there exist $w,x\in V(H)$ with $d_G(w,x)=2$ but
$d_H(w,x)>2$.  Denote the common neighbor(s) in $G$ of $w$ and $x$
by $y_1$ and $y_2$ (if it exists).
Since $d_H(w,x)>2$, we have $y_1,y_2\notin V(H)$.  Since $H$ has diameter at
most four, a short case analysis shows that $H$ contains at least four neighbors
of either $y_1$ or $y_2$.  (Since $H$ has diameter at most 4, but $G$ has
edge-width at least 7, we can essentially ignore the possibility of
non-contractible cycles creating problems.)
\end{proof}

\begin{remark}
Often when we claim, in later proofs, that a template is good, we will be
implicitly using Lemmas~\ref{good-template-lem} and~\ref{loc-help-lem}.
A \Emph{triple} is an independent set of size 3 with a
common neighbor.  For example, in each picture in Figure~\ref{4template-ex-fig}
a triple comprises the three orange neighbors of vertex 1.
A triple itself is not a good 4-template.  However, each
triple is a subset of 12 good 4-templates.  These include the two good
4-templates shown in Figure~\ref{4template-ex-fig}, along with 10 others that
arise by rotation and reflection.  Further, we show in Lemma~\ref{triple-lem}
that if $\vph$ contains a triple, then $\vph$ is 5-equivalent to some
5-coloring that contains a good 4-template.  Because triples have fewer
vertices than any good template, they are easier to work with.  Thus, we aim to
show that every 5-coloring is 5-equivalent to a 5-coloring containing a triple.
This motivates the following lemma.
\end{remark}

\begin{lem}
Let $T$ be a good template for $G$, as witnessed by a subgraph $H$ in
Lemma~\ref{good-template-lem}.  Let $v\notin V_H$ be a vertex such that
$T\cup\{v\}$ is also a good template, with $v$ as its own color.  
For each 5-coloring $\vph_0$ of $T$ and each $\alpha\in\{1,\ldots,5\}$ such that
using $\alpha$ on $v$ extends $\vph_0$ to a proper 5-coloring $\vph_0'$ of
$T+v$, there exists a proper 5-coloring of $G$ that extends $\vph_0'$.
This $v$ is called a \emph{bonus vertex}\aside{\emph{bonus vertex}} for $T$.
In particular, if there
exists some other color $C$ in $T$ such that $C\cup\{v\}$ includes a triple,
then every 5-coloring $\vph$ containing $T$ is 5-equivalent to a 5-coloring
containing a triple. 
\end{lem}

\begin{proof}
For an example, see Figure~\ref{helper-fig}, where $T$ has the two colors
$\{1,2,3,4\}$ and $\{6,9\}$, $H$ is
induced by $1,\ldots,11$ and $v=12$.
The hypothesis that $v$ is its own color implies that $v$ is not identified with
any other vertex of $T$.  Thus, $v$ can appear last in the degeneracy order, and
is not required to have the same color as any other vertex in $T$.
Since $T\cup\{v\}$ is a good template, every proper 5-coloring of $T\cup\{v\}$,
that respects its colors, extends to a 5-coloring of $G$ (we simply color the
vertices outside $T\cup\{v\}$ in the reverse of the 4-degeneracy order).  
Fix a proper 5-coloring $\vph_0'$ of 
$T\cup \{v\}$,
and let $\vph$ be an extension
of this coloring of $T\cup\{v\}$ to a 5-coloring of $G$.  Let $\vph_1$ be any
other 5-coloring of $G$ that contains $T$.  Since $\vph_1$ and $\vph$ both
contain $T$, they are 5-equivalent.  If $C\cup\{v\}$ includes a triple, then
$\vph$ contains this triple, so $\vph_1$ is
indeed 5-equivalent to a 5-coloring that contains a triple.
\end{proof}

\subsection{Shifted Triangulated Toroidal Grids}

We denote each vertex in a triangulated toroidal grid $T[a\times b]$ by an
ordered pair $(i,j)$ with $i\in\{1,\ldots,a\}$ and $j\in \{1,\ldots,b\}$;
vertices are numbered as entries of a matrix, with $(1,1)$ in the top left and
$(a,b)$ in the bottom right; see Figure~\ref{tri-toroidal-fig}.
So $(i,j)$ is adjacent to $(i,j-1),(i-1,j), (i-1,j+1), (i,j+1), (i+1,j),
(i+1,j-1)$ with arithmetic modulo $a$ and $b$, as appropriate. 
Let $T[a\times b,c]$ denote a \Emph{triangulated
$a\times b$ toroidal grid with shift $c$}.  The vertex
set is the same as that for $T[a\times b]$, and the edge set is the same except
that edges from column $b$ to column 1 are ``shifted'' by $c$.  More precisely,
each vertex $(i,b)$ is adjacent to $(i+c-2,1), (i+c-1,1), (i+1,b), (i+1, b-1),
(i,b-1), (i-1,b)$.  So $T[a\times b]=T[a\times b,1]$.  
The following useful result\footnote{This
characterization was stated by Altshuler~\cite{altshuler}, but his paper did not
give a complete proof. About 10 years later, Negami published a
proof~\cite{negami}, apparently unaware of the work of Altshuler.}
characterizes all 6-regular toroidal graphs.

\addtocounter{thm}{1}
\begin{thm}[\cite{altshuler, negami}]
Every 6-regular toroidal graph has the form $T[a\times b,c]$ for some positive
integers $a,b,c$ with $c\le a$.
\label{6reg-lem}
\end{thm}

This is Theorem~3.4~in~\cite{negami}; the condition $c\le a$ draws on the
comment in~\cite{negami} at the bottom of page 169.

A \Emph{circulant} \EmphE{\mbox{$C_n[1,r,r+1]$}}{0mm} is a 6-regular graph\footnote{When
$n\in\{2r,2r+1,2r+2\}$ it is actually a 6-regular multigraph.} with vertex set
$\{1,\ldots,n\}$ and with $i$ and $j$ adjacent if $i-j\in\{\pm 1,\pm r,\pm
(r+1)\}$.  Each circulant $C_n[1,r,r+1]$ has a natural embedding as a
triangulation of the torus: begin with the hamiltonian cycle consisting of edges
of ``length'' 1, wrapping around the torus in one direction, and now embed the
remaining edges, each wrapping around the torus in the other direction.
Building on work of Collins and Hutchinson~\cite{CH}, Yeh and Zhu~\cite{YZ}
proved the following.

\begin{thm}[\cite{YZ,CH}]
\label{6reg-4col}
All 6-regular toroidal graphs are 4-colorable, with the following exceptions:
\begin{enumerate}
\item $G\in \{T[3\times 3,2]$, $T[3\times 3,3]$, $T[5\times 3,2]$, $T[5\times
3,3]$, $T[5\times 5,3]$, $T[5\times 5,4]\}$.
\item $G=T[m\times 2,1]$ with $m$ odd.
\item $G=C_n[1,r,r+1]$ and $n\in\{2r+2,2r+3,3r+1,3r+2\}$ and $n$ is not divisible by 4.
\item $G=C_n[1,2,3]$ and $n$ is not divisible by 4.
\item $G = C_n[1,r,r + 1]$ and
$(r,n)\in\{(3,13),(3,17),(3,18),(3,25),(4,17),(6,17),(6,25), (6, 33), (7, 19)$,
$(7, 25)$, $(7, 26), (9, 25), (10, 25), (10, 26), (10, 37), (14, 33)\}$.
\end{enumerate}
\end{thm}

It is easy to prove that the graphs in Theorem~\ref{6reg-4col} are not
4-colorable.  More specifically, if $\alpha(G)$ denotes the indpendence number
of $G$, then simple arguments show that each of the graphs in parts
(1--4) have $\alpha(G)<|G|/4$, and thus, $\chi(G)\ge |G|/\alpha(G) >
4$; details are given in Section~3 of~\cite{YZ}.  
(The proof for graphs in (5) is ad hoc.)
Thus, the result in Theorem~\ref{6reg-4col} is best possible.
\begin{lem}
If $G$ is a 6-regular toroidal graph with edge-width at least 6, 
and $G\notin \{C_{26}[1,10,11]$, $C_{37}[1,10,11]\}$, then $G$ is 4-colorable.
\label{4colorable-lem}
\end{lem}
\begin{proof}
This is an easy consequence of Theorem~\ref{6reg-4col}.
Each graph in (1), (2), (3), and (4) has edge width at most 5, 3, 3, and 3
(respectively).  So we only consider graphs in (5).  Note that $C_n[1,r,r+1]$
always has edge-width at most $r+1$.  So we assume that $r\ge 6$.  When
$(r,n)=(6,17)$, we have the non-contractible cycle with successive edge lengths
$6,6,6,-1$. When $(r,n)=(6,25)$, we have $6,6,6,7$.  When $(r,n)=(6,33)$, we
have $6,6,7,7,7$.  When $(r,n)=(7,19)$, we have $7,7,7,-1,-1$.  
When $(r,n)=(7,25)$, we have $8,8,8,1$.  
When $(r,n)=(7,26)$, we have $8,8,8,1,1$.  
When $(r,n)=(9,25)$, we have $9,9,9,-1,-1$.  
When $(r,n)=(10,25)$, we have $11,11,1,1,1$.  
When $(r,n)=(14,33)$, we have $15,15,1,1,1$.  
\end{proof}

\begin{lem}
Let $G$ be a 6-regular toroidal graph with edge-width at least 7.  If $\vph_1$
and $\vph_2$ are 5-colorings of $G$ and each contains some good monochromatic
template (possibly different templates in $\vph_1$ and $\vph_2$), then $\vph_1$
and $\vph_2$ are 5-equivalent.
\label{main-helper}
\end{lem}
\begin{proof}
When $G$ is 4-colorable, the result follows from Lemma~\ref{switch-lem}.
So, by Lemma~\ref{4colorable-lem}, we only need to consider
$G\in\{C_{26}[1,10,11],C_{37}[1,10,11]\}$.  Further, $C_{26}[1,10,11]$ has a
non-contractible cycle of length 6; it has edge lengths $11,11,1,1,1,1$.  So
assume that $G=C_{37}[1,10,11]$.  Let $\vph_0$ denote the 5-coloring of $G$ that
uses color $i\bmod 4$ on vertex $i$, for each $i<37$, and uses a fifth color on
vertex 37, call it green.  An \emph{$s$-rotation} of $\vph_0$ uses green on vertex
$s$ and uses color $i\bmod 4$ on vertex $i+s$, for each $i<37$; a
\Emph{rotation} is an $s$-rotation for some value of $s$.  If $T_1$ is a
4-template, then some rotation $\vph_0^s$ of $\vph_0$ uses green only on a
vertex that is not in $\cup_{v\in T_1}N[v]$, since $|T_1|*7=28<37$; in fact,
there exist at least $37-28=9$ of these.  So there exists a 5-coloring
$\vph_{1,s}$ of $G$ that agrees with $\vph_0^s$ outside of $T_1$
and uses green on $T_1$; again, there are at least 9 of these.  
By Lemma~\ref{template-lem}, we know that all of these 5-colorings containing
$T_1$ are 5-equivalent; further, from each of them, we can recolor the vertices
of $T_1$ to reach $\vph_0^s$.  By the transitivity of
equivalence, and the fact that $G$ is vertex transitive, every 5-coloring
containing $T_1$ is 5-equivalent to every rotation of $\vph_0$.  The same is
true of 5-colorings containing $T_2$.  Thus, $\vph_1$ and $\vph_2$ are
5-equivalent.
\end{proof}

Now we prove the Main Theorem, assuming results in
Section~\ref{templates-sec}.  For reference, we restate it.

\begin{mainthm}
If $G$ is a 6-regular toroidal graph with edge-width at least 7, then all
5-colorings of $G$ are 5-equivalent.
\end{mainthm}
\begin{proof}
Let $G$ be a 6-regular toroidal graph with edge-width at least 7.  In
Lemma~\ref{main-lem} of Section~\ref{templates-sec} we prove that every
5-coloring of $G$ is
5-equivalent to a 5-coloring that contains a good 4-template.  Now
Lemma~\ref{main-helper} proves that all 5-colorings of $G$ are 5-equivalent.
\end{proof}

An event $E(n)$ happens \emph{asymptotically almost surely} if $\Pr(E(n))\to
1$ as $n\to \infty$.

\begin{lem}
Fix a positive integer $g$.  If $G$ is a 6-regular toroidal graph on $n$
vertices chosen uniformly at random, then asymptotically almost surely the
edge-width of $G$ is at least $g$.
\end{lem}
\begin{proof}
Fix a positive integer $n$.  We prove an upper bound on the fraction of
6-regular toroidal $n$-vertex graphs that have edge-width less than $g$,
and we show that this fraction tend to 0 as $n$ tends to infinity.
\begin{clm}
Each 6-regular toroidal graph can be written as $T[a\times b, c]$ for at most
6 choices of positive integers $(a,b,c)$, such that $c\le a$.
\end{clm}
Claim~1 follows immediately from Corollary~3.7 of~\cite{negami}.
The idea of that proof is to start from a vertex and follow edges along a
``straight line'' until returning to that vertex.  More formally, to follow a
straight line, the successor of each vertex $v$ is 3 later than its predecessor
in the cyclic order of neighbors of $v$.  The factor 6~in Claim~1 arises from the 6 choices
of neighbors to leave our initial vertex.  (Since each graph $T[a\times b,c]$ is
vertex-transitive, our choice of initial vertex has no effect.)

\begin{clm}
There exists a constant $C_g$ such that, for all positive integers $a$ and $b$,
at most $C_g$ graphs of the form $T[a\times b,c]$ (with $c\le a$) have edge-width less than $g$.
\end{clm}
Fix positive integers $a$ and $b$.  We bound (independently of $a$ and $b$) the
number of integers $c$ with $1\le c\le a$ such that $T[a\times b,c]$ has
edge-width less than $g$.  We make no effort to optimize this bound, only to
show that it is independent of $a$ and $b$.  Suppose that $T[a\times b,c]$ has
a non-contractible cycle $C$ of length less than $g$. 
Let $r$ denote the number of number of times that $C$ wraps around the torus in
the direction of the dimension of length $a$.  Let $s$ denote the number of
edges on $C$ of length 1 in this direction, and let $t_1$ and $t_2$ denote the
numbers of edges on $C$ from column $b$ to column 1 of lengths $c$ and $c+1$,
respectively.  More precisely, let $s$, $t_1$, $t_2$ denote the net numbers of
these edges traversed in the positive direction as we traverse $C$ (for some
arbitrary orientation).  So $s+ct_1+(c+1)t_2 = ra$.  Thus, $c = (ra - s -
t_2)/(t_1+t_2)$.  Since $s$, $t_1$, $t_2$, and $r$ are all bounded in terms of
$g$ (e.g.  $g>s+t_1+t_2>-g$ and $g+1\ge r\ge -(g+1)$), so is the
number of choices for $c$.  This proves the claim.

Let $d(n)$ denote the number of divisors of $n$.  To build a 6-regular toroidal
graph $T[a\times b,c]$ on $n$ vertices, we have $d(n)$ choices for the pair
$(a,b)$.  So the number of 6-regular toroidal graphs with edge-width less than
$g$ is at most $d(n)C_g$, by Claim~2.  The number of 6-regular toroidal graphs
for the pair $(a,b)$ is precisely $a$, since $1\le c\le a$.  So the number for
the pairs $(a,b)$ and $(b,a)$ is $a+b\ge \max(a,b)\ge \sqrt{n}$.  By Claim~1,
this overcounts the total number of 6-regular toroidal graphs on $n$ vertices
by a factor of at most 6.  Thus, the fraction of 6-regular toroidal graphs on
$n$ vertices with edge-width less than $g$ is at most
$d(n)C_g/(((d(n)/2)\sqrt{n})/6)=12C_g/\sqrt{n}$, which tends to 0 as $n$ tends
to infinity.
\end{proof}

This theorem and lemma immediately imply the following corollary.

\begin{maincor}
If $G$ is a 6-regular toroidal graph on $n$ vertices chosen uniformly at
random, then asymptotically almost surely all 5-colorings of $G$ are
5-equivalent.
\end{maincor}

\section{Good Templates}
\label{templates-sec}
Throughout this section we assume that $G$ is a 6-regular toroidal graph with
edge-width at least 7.
A coloring $\vph$ of $G$ is \Emph{nice} if it is 5-equivalent to some
coloring that contains a good 4-template.  In this section we prove that
every 5-coloring $\vph$ of $G$ is nice.  Our proof uses a sequence of lemmas,
with increasingly weaker hypotheses, culminating in the desired result.
We adopt conventions common to all figures in this section.  We
generally number vertices in the order that we determine their colors.  
Sometimes, we do not determine the color of a vertex, but rather determine a
color not used there.  For such vertices, we restart the numbering,
using higher numbers.
To denote that a vertex is not colored green, we lightly shade it green.  

We often argue by symmetry.  For example, we assume that any triple centered at
30 in Figure~\ref{triple-fig} is $\{1,2,3\}$, rather than $\{28,31,32\}$.  It is
possible that the only automorphism of $G$ that fixes vertex 30 is the identity
map.  So we do not quite mean that we can apply the same argument to $G$ under
some automorphism.  But we do mean something close.  If we consider the
6-regular triangulation of the plane, the set of automorphisms that fix a given
vertex $v$ has size 12; it has 6-fold rotational symmetry, together with
reflection through a line containing $v$.  The essential point is that our
arguments do not make use of the global structure of $G$, but only of subgraphs
of $G$ of diameter at most 6.  Each time that we invoke symmetry, we appeal to a
map from one such subgraph to another, typically with many vertices in common.
And we always mean one of these 12 symmetries mentioned above.
To highlight these symmetries, we draw our pictures as subgraphs of the
6-regular triangulation of the plane.  However, a moment's reflection shows that
these are isomorphic to subgraphs of a triangulated toroidal grid 
(as it is drawn in Figure~\ref{tri-toroidal-fig}).

Along similar lines, it is possible that two vertices drawn as distinct in some
figure are in
fact the same vertex, when $G$ has a short non-contractible cycle.  Such a pair
must be drawn at distance at least 7.  Similarly, vertices may be joined by an
edge that is not shown; but such vertices must be drawn at distance at least 6.
In some figures, the subgraph shown may have diameter 7 or more.  However, the
attentive reader will note that whenever we conclude that $\vph$ contains a good
template $T$, the vertices of $T$ are contained in some subgraph $H$ with
diameter at most 4.  Thus, we avoid these potential complications.

\begin{lem}
If a 5-coloring $\vph$ of $G$ contains a triple, then $\vph$ is nice.
\label{triple-lem}
\end{lem}

\noindent
\textit{Proof.}
Assume the lemma is false, and let $\vph$ be a counterexample with a triple
centered at 30, as in Figure~\ref{triple-fig}.

(1-3) By symmetry, assume 1, 2, 3 are orange.  Neither 6 nor 7 is orange; 
otherwise $\{1,2,3,6\}$ or $\{1,2,3,7\}$ is a good template.  By symmetry,
none of the following are orange: 9, 10, 12, 13, 15, 16, 18, 19, 21, 22.

(4) We will show that 8 is not orange; suppose instead that it is.  If 23 or 24
is orange,
then $\vph$ has a good
orange template at $\{1,2,8,w\}$, where $w\in\{23,24\}$.  So neither 23 nor 24
is orange;  similarly, neither 25 nor 26 is orange.  Let $\alpha$ denote the
color of 7.  If neither 28 nor 29 is $\alpha$, then an $\alpha$/orange swap at
7,8 makes 7
orange (and does not change the color anywhere else except 8, and possibly 9).  
This gives a good template at $\{1,2,3,7\}$, which shows that 8 is not orange. 
So it suffices to ensure that neither 28 nor 29 is $\alpha$. 

First suppose that 28 is $\alpha$.  If 27, 29, and 30 do not have distinct
colors, then some color $\beta$ is absent from the closed neighborhood of 28,
so we use an $\alpha$/$\beta$ swap at 28, ensuring that neither 28 nor 29 is
$\alpha$.  Thus, we assume that 27, 29, and 30 have distinct colors.
If 28, 31, 32 have distinct colors, then let $\gamma$ denote the color on 30.
Now an $\alpha$/$\gamma$ swap at 28,30 reduces to the case where 28 (and also
29) is not color
$\alpha$, handled above.  So assume that 28, 31, 32 do not have distinct colors.
Thus, some color is absent from the closed neighborhood of 30.  Now recoloring
30 ensures that 30 has the same color as 27 or 29, which we have handled above.
Thus, we assume that 28 is not $\alpha$.

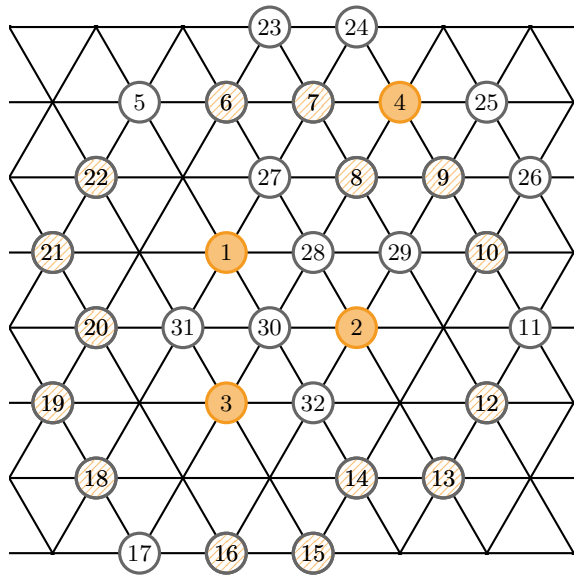
\begin{wrapfigure}{l}{0.55\textwidth}
\centering
\begin{tikzpicture}
\begin{scope}
\clip(-3.46,-3.02) rectangle (4.04,4.02);  
\foreach \ang in {0, 60, 120}  
{
\begin{scope}[rotate=\ang]
\foreach \i in {-6,...,5} 
\draw[thick] (-6.5, \i) -- (6.5,\i);
\end{scope}
}
\end{scope}
\tikzstyle{uStyle}=[shape = circle, minimum size = 15pt, inner sep =
1.5pt, outer sep = 0pt, draw, fill=white]

\begin{scope}[xscale=.577] 
\foreach \i/\j/\lab/\col in {
-1/1/1/orange, 
2/0/2/orange, 
-1/-1/3/orange, 
3/3/4/orange}
\draw[very thick] (\i,\j) node[uStyle, draw=my\col, fill=my\col!60] {\small{\lab}};  

\foreach \i/\j/\lab/\col in {
5/3/25/gray, 6/2/26/gray, 0/4/23/gray, 2/4/24/gray, 
0/2/27/gray, 1/1/28/gray, 3/1/29/gray, 
0/0/30/gray, -2/0/31/gray, 1/-1/32/gray, 
-3/3/5/gray, 6/0/11/gray, -3/-3/17/gray}
\draw[very thick] (\i,\j) node[uStyle, draw=black!60!white, fill=white] {\small{\lab}};

\foreach \i/\j/\lab/\col in {
-1/3/6/orange, 1/3/7/orange, 2/2/8/orange, 4/2/9/orange,
5/1/10/orange, 5/-1/12/orange, 4/-2/13/orange, 2/-2/14/orange,
1/-3/15/orange, -1/-3/16/orange, -4/-2/18/orange, -5/-1/19/orange,
-4/0/20/orange, -5/1/21/orange, -4/2/22/orange}
\draw[very thick] (\i,\j) node[uStyle, draw=black!60!white, fill=white] {\small{\lab}}
(\i,\j) node[uStyle, draw=black!60!white, pattern=north east
lines, pattern color=my\col!60!white] {\small{\lab}};
\end{scope}
\end{tikzpicture}
\caption{The proof of
Lemma~\ref{triple-lem}.\label{triple-fig}}

\end{wrapfigure}

Suppose instead that 29 is $\alpha$.  By symmetry, we can assume that 9 and 27
are both some color $\beta$; otherwise, we reflect across the line through 8 and
28.  Now we will recolor 8 so that it is not orange.
If 28 and 4 have a common color, then some color is unused on the closed
neighborhood of 8, so we can recolor 8.  So assume that 28 and 4 use distinct
colors; let $\gamma$ be the color on 28.
If 30 is $\alpha$ or $\beta$, then some color is unused on the closed
neighborhood of 28, so we can recolor 28 using the color on 4. This allows us
to recolor 8, which ensures that neither 8 nor 4 is orange.  Now we can recolor
7 orange, reducing to a case above.
So assume that 30 is neither $\alpha$ nor $\beta$.
If $\gamma$ does not appear on $\{31,32\}$, then we swap colors
on 28 and 30 and recolor 8 with $\gamma$.
So assume that $\gamma$ appears on $\{31,32\}$.  Now some color is unused on the
closed neighborhood of 30, so we can recolor 30, ensuring that it uses either
$\alpha$ or $\beta$; this reduces to a case above.
Thus, we assume that 29 is not $\alpha$.
All of this shows that we can assume that 8 is not orange.
By symmetry, neither 14 nor 20 is orange.  So 4 is orange; otherwise we 
recolor 8 with orange and reduce to a previous case.

By symmetry, assume that green does not appear on the unique
common neighbor of any of the sets of vertices $\{1,5\}$, $\{2,11\}$, $\{3,17\}$.
We use green/orange swaps at each green/orange component containing
vertices 1, 2, and 3.  In the resulting coloring, 1, 2, and 3 are green but
4 is orange.  So either one of 7, 8, or 9 is green or we can recolor 8 with
green.  For each possibility, we reduce to an earlier case. \hfill$\qed$

\bigskip

Let \Emph{parallel pairs} denote the template, with two colors and two vertices in
each color, induced by vertices $\{1, 2,3,4\}$ in Figure~\ref{parallel-pairs-fig}.

\begin{lem}
If a 5-coloring $\vph$ of $G$ contains parallel pairs, then $\vph$ is nice. 
\label{parallel-pairs-lem}
\end{lem}
\noindent\textit{Proof.}
Assume the lemma is false, and let $\vph$ be a counterexample.  By
Lemma~\ref{triple-lem}, it suffices to show that $\vph$ is equivalent to a
5-coloring that contains a triple.  We assume 1,2 are green and 3,4 are blue,
as in Figure~\ref{parallel-pairs-fig}.

(5-7) Note that 5 must be green, as follows.  If some other neighbor $v$ of 19
is green, then $\vph$ contains a good template at $\{v,1,2,3,4\}$.
The 4-degeneracy prefix begins 20, 14 and the bonus vertex is 15 (which we make
green). 
If 19 has no green neighbor, then we can recolor 19 green and get a green triple
at $\{1,2,19\}$.  So 5 is green, as claimed.  By the same argument, 6 is blue.
If 10 is green, then $\vph$ contains a good template at $\{1,2,3,4,10\}$,
with prefix 20, 14, 21 and bonus vertex 15.  So 10 is not green; by symmetry,
neither is 11.  Similarly, neither
12 nor 13 is blue.  Suppose that none of 7, 8, 16, and 25 is green or blue.
Now we use green/blue swaps at 1,3 and at 2,4.
Each of these green/blue components that we recolored contains at most 4
vertices; in particular, neither contains 5, so 5 is still green.  If 19 is
blue, then we have a blue triple centered at 14.  If a neighbor $v$ of 19 other
than 5 is blue, then we have a good template $\{v,1,2,3,4\}$, as above (with 15
as bonus vertex).  Otherwise, we can recolor 19 blue to get a blue triple
centered at 14.  Thus, we assume that at least one of 7, 8, 16, and 25 is blue
or green.  By symmetry, assume that 7 is blue.

(8) 
Note that 14 is not blue, since then we have a blue triple at $\{3,4,14\}$.
Similarly, 15 is not green.
Suppose, to reach a contradiction, that 16 is blue.  Let $\alpha$ be a color
that is not green and not blue and that does not appear on 23
or 24.  We use an $\alpha$/green swap at 1 and an $\alpha$/green swap at 2 (possibly
a single swap, if 1 and 2 are in the same $\alpha$/green component).  It is easy
to check that 1 and 2 are the only green vertices that get recolored $\alpha$.
In particular, 5 is still green.  

If a neighbor $v$ of 19 is
$\alpha$, then we have a good template at $\{v,1,2,3,4\}$.  Otherwise, either 19
is $\alpha$ or we can recolor it $\alpha$.  In each case, we get an $\alpha$
triple at $\{1,2,19\}$.  Thus, 16 is not blue.
So 8 must be green.  Otherwise, the blue/green
component containing 2,4 has at most two more vertices: 21 and 22.  After a
blue/green swap at 2,4, we can recolor 21 green.  This gives a good template
at $\{1,2,3,4,5,21\}$, with bonus vertex 15.  Thus, 8 is green.  

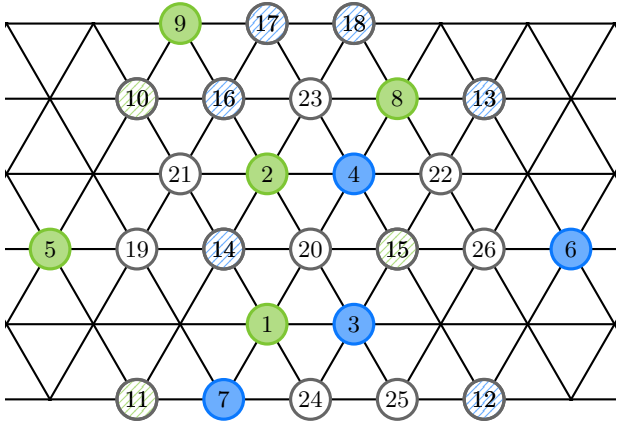
\begin{wrapfigure}{l}{0.50\textwidth}
\centering
\begin{tikzpicture}
\begin{scope}
\clip(-4.05,-2.02) rectangle (4.05,3.02);  
\foreach \ang in {0, 60, 120}  
{
\begin{scope}[rotate=\ang]
\foreach \i in {-5,...,5} 
\draw[thick] (-5.5, \i) -- (6,\i);
\end{scope}
}
\end{scope}
\tikzstyle{uStyle}=[shape = circle, minimum size = 15pt, inner sep =
1.5pt, outer sep = 0pt, draw, fill=white]

\begin{scope}[xscale=.577] 
\foreach \i/\j/\lab/\col in {
-1/-1/1/green, -1/1/2/green, 1/-1/3/blue, 1/1/4/blue,
-6/0/5/green, 6/0/6/blue, -2/-2/7/blue,
2/2/8/green, -3/3/9/green}
\draw[very thick] (\i,\j) node[uStyle, draw=my\col, fill=my\col!60] {\small{\lab}};  

\foreach \i/\j/\lab in {
-4/0/19, 
0/0/20,
-3/1/21, 
3/1/22, 
0/2/23, 
0/-2/24, 
2/-2/25, 
4/0/26} 
\draw[very thick] (\i,\j) node[uStyle, draw=black!60!white, fill=white] {\small{\lab}};

\foreach \i/\j/\lab/\col in {-4/-2/11/green, -4/2/10/green, 
2/0/15/green, 
-2/0/14/blue, 4/-2/12/blue, 4/2/13/blue, -2/2/16/blue, -1/3/17/blue, 1/3/18/blue}
\draw[very thick] (\i,\j) node[uStyle, draw=black!60!white, fill=white] {\small{\lab}}
(\i,\j) node[uStyle, draw=black!60!white, pattern=north east
lines, pattern color=my\col!60!white] {\small{\lab}};

\end{scope}
\end{tikzpicture}
\caption{The proof of
Lemma~\ref{parallel-pairs-lem}.\label{parallel-pairs-fig}}
\end{wrapfigure}

(9) Note that 17 is not green (or blue, as we will see); otherwise, it would
form a green triple with 2 and 8.  Suppose that 9 is not green.  Let $\alpha$
be a color that is not green or blue and is also not used on 23 or 24.  Now we
use a green/$\alpha$ swap at 1 and also use a green/$\alpha$ swap at 2.  It is
easy to check that the only green vertices that get uncolored are 1 and 2.  In
particular, 5 is still green.  If a neighbor $v$ of 19 is colored $\alpha$,
then we have a good template at $\{v,1,2,3,4\}$, with bonus vertex 15.  So
either 19 is $\alpha$ or we can recolor it $\alpha$.  In either case, we get a
triple colored $\alpha$.  Thus, 9 is green.  

If 17 is blue, then we have a good template at $\{1,2, 3, 4, 8,9,17\}$, with
bonus vertex 15;
the degeneracy prefix is 23, 16, 20, 14, 21.  So 17 is
not blue.  Similarly, 18 is not blue; this good template replaces 17 in the
previous one with 18, and the degeneracy prefix is 23, 16, 17, 20, 14, 21.
Let $\beta$ be a color that is not green or blue and is unused on 24 and 25. 
We use $\beta$/blue swaps at 3 and at 4.  It is easy to check that the only
blue vertices that get uncolored are 3 and 4.  This new coloring again
satisfies the hypotheses of the lemma, but 3 and 4 have a common color $\beta$
different from the color on 6.  We repeat our argument above that showed that 6
was blue.  Now it shows that 6 is $\beta$, but in fact 6 is blue.
This contradiction finishes the proof.
\hfill$\qed$%

\begin{lem}
If a 5-coloring $\vph$ of $G$ uses the same color on vertices 1, 2, 3, 4 in
Figure~\ref{helper-fig}, then $\vph$ is nice.
\label{helper-lem}
\end{lem}

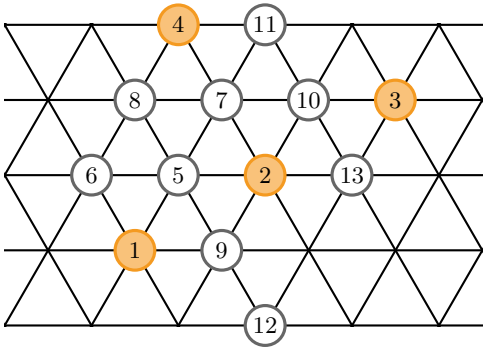
\begin{wrapfigure}{l}{0.5\textwidth}
\centering
\begin{tikzpicture}
\begin{scope}
\clip(-1.73,-1.02) rectangle (4.62,3.02);  
\foreach \ang in {0, 60, 120}  
{
\begin{scope}[rotate=\ang]
\foreach \i in {-5,...,5} 
\draw[thick] (-5.5, \i) -- (5.5,\i);
\end{scope}
}
\end{scope}
\tikzstyle{uStyle}=[shape = circle, minimum size = 15pt, inner sep =
1.5pt, outer sep = 0pt, draw, fill=white]

\begin{scope}[xscale=.577] 
\foreach \i/\j/\lab/\col in {
0/0/1/orange,
3/1/2/orange,
6/2/3/orange,
1/3/4/orange} 
\draw[very thick] (\i,\j) node[uStyle, draw=my\col, fill=my\col!60!white]
{\small{\lab}};  

\foreach \i/\j/\lab/\col in {1/1/5, -1/1/6, 2/2/7, 2/0/9,
0/2/8, 4/2/10, 3/3/11, 3/-1/12, 5/1/13}
\draw[very thick] (\i,\j) node[uStyle, draw=black!60!white, fill=white] {\small{\lab}};
\end{scope}
\end{tikzpicture}
\caption{\label{helper-fig}%
The proof of Lemma~\ref{helper-lem}.}

\end{wrapfigure}
\noindent\textit{Proof.}
By symmetry, assume that 1, 2, 3, 4 are orange.  By Pigeonhole, some color other
than orange is used on at least two neighbors of 5; call this color blue.  If
these neighbors are 6 and 7, then $\vph$ contains parallel pairs centered at 5,
so we are done by Lemma~\ref{parallel-pairs-lem}.  Thus, we assume that blue is
used on vertices 9 and $v$, where $v\in\{6,7,8\}$.  If $v=8$, then 
we have the good template $\{1,2,3,4,8,9\}$; the prefix is 5, 7, 10, 11, and the
bonus vertex is 12.  If $v=6$, then we have the good template $\{1,2,3,4,6,9\}$;
the prefix is 5, 7, 8, 10, 11 and the bonus vertex is 12.  It is easy to check
that the subgraphs induced by $\{1,\ldots,11\}$ and $\{1,\ldots,12\}$ are
well-behaved, since each of them has diameter 4, but every
non-contractible cycle has length at least 7.  So we assume that blue is used on
7 and 9.  Repeating the same argument, with 10 in place of 5, shows that blue is
used on 7 and 13.  But now $\{7,9,13\}$ is a blue triple, so $\vph$ is nice.
\hfill$\qed$%
\bigskip
\bigskip
\bigskip

Let \Emph{crossing pairs} denote the template, with two colors and two vertices in
each color, induced by vertices $\{1, 2,3,4\}$ in
Figure~\ref{crossing-pairs-fig1}.

\begin{lem}
If a 5-coloring $\vph$ of $G$ contains crossing pairs, then $\vph$ is nice.
\label{crossing-pairs-lem}
\end{lem}
\noindent
\textit{Proof.} 
We assume the lemma is false, and let $\vph$ be a counterexample.
Further, we assume that 1 is blue, 2 is red, 3 is blue, and 4 is red. 
We may also assume
that 5 is green, 6 is orange, and 7 is purple.  By symmetry between red and
blue, either 8 is green or 8 is blue.

\textbf{Case 1: 8 is green.}
Note that 9 is not red, or we have red/green parallel pairs centered at 7.
Similarly, 11 is not blue. 
Neither 10 nor 12 is green, since $\vph$ has no green triple.
 So 9 and 10 are orange and blue (in some order), and 11 and 12 are red and
purple (in some order).  

\begin{wrapfigure}{l}{0.5\textwidth}
\centering
\begin{tikzpicture}
\begin{scope}
\clip(-1.73,-1.02) rectangle (4.62,3.02);  
\foreach \ang in {0, 60, 120}  
{
\begin{scope}[rotate=\ang]
\foreach \i in {-5,...,5} 
\draw[thick] (-5.5, \i) -- (5.5,\i);
\end{scope}
}
\end{scope}
\tikzstyle{uStyle}=[shape = circle, minimum size = 15pt, inner sep =
1.5pt, outer sep = 0pt, draw, fill=white]

\begin{scope}[xscale=.577] 
\foreach \i/\j/\lab/\col in {
0/0/1/blue, 
1/-1/2/red, 
3/-1/3/blue, 4/0/4/red,
2/0/5/green, 
1/1/6/orange,
3/1/7/purple, 
2/2/8/green, 
4/2/9/orange, 
5/1/10/blue, 
0/2/11/purple,
-1/1/12/red} 
\draw[very thick] (\i,\j) node[uStyle, draw=my\col, fill=my\col!60!white]
{\small{\lab}};  

\foreach \i/\j/\lab/\col in 
{-2/2/13/orange, 
-1/3/14/orange, 
1/3/15/orange,
3/3/16/purple,
5/3/17/purple, 
6/2/18/purple}
\draw[very thick] (\i,\j) node[uStyle, draw=black!60!white, fill=white] {\small{\lab}}
(\i,\j) node[uStyle, draw=black!60!white, pattern=north east
lines, pattern color=my\col!60!white] {\small{\lab}};

\foreach \i/\j/\lab/\col in {7/1/19/gray, 6/0/20/gray}
\draw[very thick] (\i,\j) node[uStyle, draw=black!60!white, fill=white] {\small{\lab}};
\end{scope}
\end{tikzpicture}
\caption{\label{crossing-pairs-fig1}%
Case 1.1: 9 is orange and 10 is blue.} 
\end{wrapfigure}
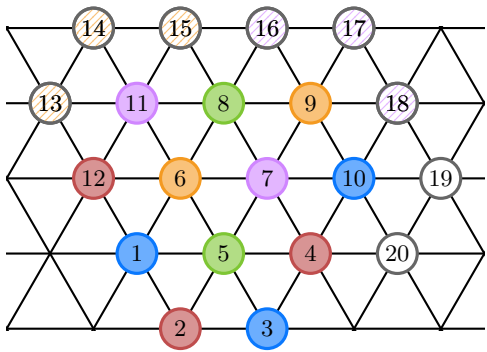

\textbf{Case 1.1: 9 is orange, 10 is blue.}
(By horizontal symmetry, and permuting colors, this also includes the
case that 9 is blue and 10 is orange, but 11 is purple and 12 is red.)
None of 13, 14, 15 is orange; otherwise we have an orange triple centered at 8
(when 15 is orange) or we have a good template.
When 13 is orange, the template is $\{1,2,3,4,6,9,13\}$,
the prefix is 5, 7, 8, 11, 12, and the bonus vertex is the
neighbor of 2 that forms a triple with 1 and 3.  When 14 is orange, the
template is $\{1,2,3,4,6,9,12,14\}$ (with 12 in its own color class) the prefix
is 5, 7, 8, 11; and the bonus vertex\footnote{We must include 12 in the first
template; otherwise, adding the bonus vertex $w$ creates a subgraph that is not
locally connected, since $d_G(11,w)=2$, but $d_{H+w}(11,w)>2$.} is the neighbor
$w$ of 12 and 1 other than 6.  After we color $w$ orange, we are done
by Lemma~\ref{helper-lem}.
If 11 is red and 12 is purple, then a red/orange swap at 11,6 gives a red triple
centered at 5.  So 11 is purple and 12 is red.  By horizontal symmetry, none of
16, 17, 18 is purple.
One of 19 and 20 is purple; otherwise a purple/blue swap at 7,10 gives a blue
triple centered at 5.  Similarly, one of 19 and 20 is orange; otherwise we use a
purple/orange swap at 11,6,7,9 followed by a blue/orange swap at 10,7, which
gives a blue triple centered at 5.  So 19 and 20 are purple and orange (in some
order).  Now either we have purple/orange parallel pairs centered at 10, or we
get them after a purple/orange swap at 11,6,7,9.
This contradicts Lemma~\ref{parallel-pairs-lem}.

\begin{wrapfigure}{l}{0.5\textwidth}
\centering
\begin{tikzpicture}
\begin{scope}
\clip(-1.73,-1.02) rectangle (4.62,3.02);  
\foreach \ang in {0, 60, 120}  
{
\begin{scope}[rotate=\ang]
\foreach \i in {-5,...,5} 
\draw[thick] (-5.5, \i) -- (5.5,\i);
\end{scope}
}
\end{scope}
\tikzstyle{uStyle}=[shape = circle, minimum size = 15pt, inner sep =
1.5pt, outer sep = 0pt, draw, fill=white]

\begin{scope}[xscale=.577] 
\foreach \i/\j/\lab/\col in {
0/0/1/blue, 
1/-1/2/red, 
3/-1/3/blue, 4/0/4/red,
2/0/5/green, 
1/1/6/orange,
3/1/7/purple, 
2/2/8/green, 
4/2/9/blue, 
5/1/10/orange, 
0/2/11/red,
-1/1/12/purple, 
1/3/13/purple,
3/3/14/orange}
\draw[very thick] (\i,\j) node[uStyle, draw=my\col, fill=my\col!60!white]
{\small{\lab}};  



\foreach \i/\j/\lab/\colA/\colB in 
{-2/2/15/blue/orange, 
-1/3/16/blue/orange, 
5/3/17/red/purple, 
6/2/18/red/purple}
{
    \draw[very thick, my\colA, fill=my\colA!60!white] (\i,\j) circle (14pt and 8pt); 
    \fill[my\colB!60!white] (\i,\j) --+ (90:8pt) arc (90:270:14pt and 8pt) -- cycle; 
    \draw[very thick, my\colB] (\i,\j) ++ (90:8pt) arc (90:270:14pt and 8pt); 
    \draw (\i,\j) node[uStyle, draw=none, fill=none] {\small{\lab}};
}

\end{scope}
\end{tikzpicture}
\caption{\label{crossing-pairs-fig2}%
Case 1.2: 9 is blue, 10 is orange, 11 is red, 12 is purple.}
\end{wrapfigure}
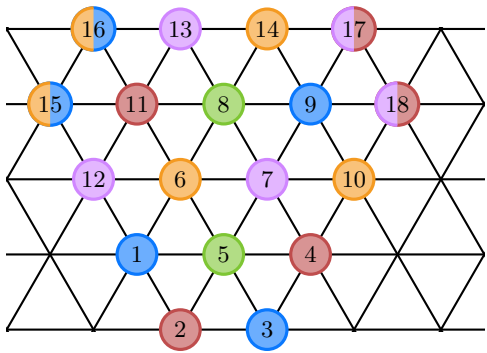

\textbf{Case 1.2: 9 is blue, 10 is orange, 11 is red, 12 is purple.}
One of 13 and 14 is orange; otherwise a green/orange swap at 5,6,8 gives an
orange triple centered at 7.  Similarly, one of 13 and 14 is purple; otherwise,
a purple/green swap at 5,7,8 gives a purple triple centered at 6.
By Lemma~\ref{parallel-pairs-lem}, $\vph$ has no parallel pairs, so 13 is
purple and 14 is orange.  If neither 15 nor
16 is orange, then an orange/red swap at 6,11 gives a red triple centered at 5.
 If neither 15 nor 16 is blue, then we recolor 11 blue, and recolor 6
red, which gives a red triple centered at 6.  So 15 and 16 are orange and blue
(in some order).  By symmetry, 17 and 18 are purple and red (in some order).
Now we use a blue/green swap at 8,9, recolor 11 green, and recolor 6 red; this
gives a red triple centered at 5.

\indent\textbf{Case 2: 8 is blue.}
Now 9 is not red; otherwise, we have a good template
at $\{1,2,3,4,9\}$, with prefix 5,7, and the bonus vertex is the neighbor of 2
that forms a triple with 1 and 3.  So 9 is either orange or green.

\textbf{Case 2.1: 9 is orange.}
(10-12) If 10 is green, then we have green/orange parallel pairs
centered at 7.  So 10 is blue.  None of 15, 16, 17 is orange; otherwise we
have an orange triple centered at 8 or a good template 
as in Case 1. (If $v=15$, then our bonus vertex gives a blue triple. But if
$v=16$, then we use Lemma~\ref{helper-lem}, as in Case~1.)  

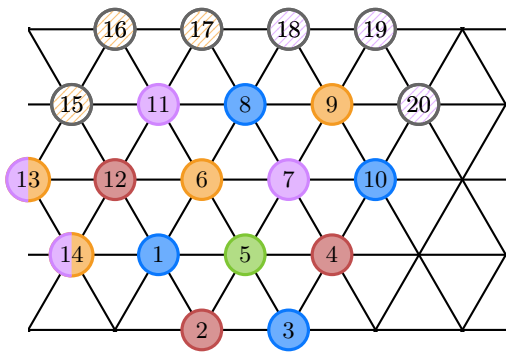
\begin{wrapfigure}{l}{0.5\textwidth}
\centering
\begin{tikzpicture}
\begin{scope}
\clip(-1.73,-1.02) rectangle (4.62,3.02);  
\foreach \ang in {0, 60, 120}  
{
\begin{scope}[rotate=\ang]
\foreach \i in {-5,...,5} 
\draw[thick] (-5.5, \i) -- (5.5,\i);
\end{scope}
}
\end{scope}
\tikzstyle{uStyle}=[shape = circle, minimum size = 15pt, inner sep =
1.5pt, outer sep = 0pt, draw, fill=white]

\begin{scope}[xscale=.577] 
\foreach \i/\j/\lab/\col in {
0/0/1/blue, 
1/-1/2/red, 
3/-1/3/blue, 4/0/4/red,
2/0/5/green, 
1/1/6/orange,
3/1/7/purple, 
2/2/8/blue, 
4/2/9/orange, 
5/1/10/blue, 
0/2/11/purple,
-1/1/12/red} 
\draw[very thick] (\i,\j) node[uStyle, draw=my\col, fill=my\col!60!white]
{\small{\lab}};  

\foreach \i/\j/\lab/\col in 
{-2/2/15/orange, 
-1/3/16/orange, 
1/3/17/orange,
3/3/18/purple,
5/3/19/purple, 
6/2/20/purple}
\draw[very thick] (\i,\j) node[uStyle, draw=black!60!white, fill=white] {\small{\lab}}
(\i,\j) node[uStyle, draw=black!60!white, pattern=north east
lines, pattern color=my\col!60!white] {\small{\lab}};

\foreach \i/\j/\lab/\colA/\colB in {
-3/1/13/orange/purple, -2/0/14/orange/purple}
{
    \draw[very thick, my\colA, fill=my\colA!60!white] (\i,\j) circle (14pt and 8pt); 
    \fill[my\colB!60!white] (\i,\j) --+ (90:8pt) arc (90:270:14pt and 8pt) -- cycle; 
    \draw[very thick, my\colB] (\i,\j) ++ (90:8pt) arc (90:270:14pt and 8pt); 
    \draw (\i,\j) node[uStyle, draw=none, fill=none] {\small{\lab}};
}
\end{scope}
\end{tikzpicture}
\caption{\label{crossing-pairs-fig3}%
Case 2.1: 9 is orange.} 
\end{wrapfigure}

If 11
is red, then a red/orange swap at 11,6 gives a red triple centered at 5.  So
12 is red; otherwise, we recolor 6 red, getting a red triple centered at
5.  And 11 is green or purple.  Assume 11 is purple; otherwise, a green/purple
swap at 5,7 reduces to this case (with green and purple interchanged).  
Similar to 15, 16, 17, none of 18, 19, 20 is purple.  
(13-14) Now 13 or 14 is
orange; otherwise an orange/red swap at 12,6 gives a red triple centered at 5. 
Also, 13 or 14 is purple; otherwise a purple/orange swap at 11,6,7,9, followed
by a red/purple swap at 12,6, again gives a red triple centered at 5.  
So 13 and 14 are orange and purple (in some
order).  But now we either have orange/purple parallel pairs centered at 12, or
else we get them after a purple/orange swap at 11,6,7,9.

\textbf{Case 2.2: 9 is green.}
(11-12) Now 11 or 12 is red; otherwise, we recolor 6 red and get a red triple centered
at  5.  And 11 or 12 is green; otherwise a green/orange swap at 5,6 reduces to
case 2.1, with green and orange interchanged.
Note that 10 is either blue or orange.  

\begin{wrapfigure}{l}{0.5\textwidth}
\centering
\begin{tikzpicture}
\begin{scope}
\clip(-1.73,-1.02) rectangle (4.62,3.02);  
\foreach \ang in {0, 60, 120}  
{
\begin{scope}[rotate=\ang]
\foreach \i in {-5,...,5} 
\draw[thick] (-5.5, \i) -- (5.5,\i);
\end{scope}
}
\end{scope}
\tikzstyle{uStyle}=[shape = circle, minimum size = 15pt, inner sep =
1.5pt, outer sep = 0pt, draw, fill=white]

\begin{scope}[xscale=.577] 
\foreach \i/\j/\lab/\col in {
0/0/1/blue, 
1/-1/2/red, 
3/-1/3/blue, 4/0/4/red,
2/0/5/green, 
1/1/6/orange,
3/1/7/purple, 
2/2/8/blue, 
4/2/9/green, 
6/2/16/red,
7/1/15/orange, 
6/0/14/purple,
5/-1/13/orange,
5/1/10/blue} 
\draw[very thick] (\i,\j) node[uStyle, draw=my\col, fill=my\col!60!white]
{\small{\lab}};  

\foreach \i/\j/\lab/\colA/\colB in {
0/2/11/green/red, 
-1/1/12/green/red,
5/3/17/orange/purple,
3/3/18/orange/purple, 
1/3/19/green/red,
-2/2/20/orange/purple, 
-1/3/21/orange/purple}
{
    \draw[very thick, my\colA, fill=my\colA!60!white] (\i,\j) circle (14pt and 8pt); 
    \fill[my\colB!60!white] (\i,\j) --+ (90:8pt) arc (90:270:14pt and 8pt) -- cycle; 
    \draw[very thick, my\colB] (\i,\j) ++ (90:8pt) arc (90:270:14pt and 8pt); 
    \draw (\i,\j) node[uStyle, draw=none, fill=none] {\small{\lab}};
}

\end{scope}
\end{tikzpicture}
\caption{\label{crossing-pairs-fig4}%
Case 2.2a: 10 is blue.} 
\end{wrapfigure}
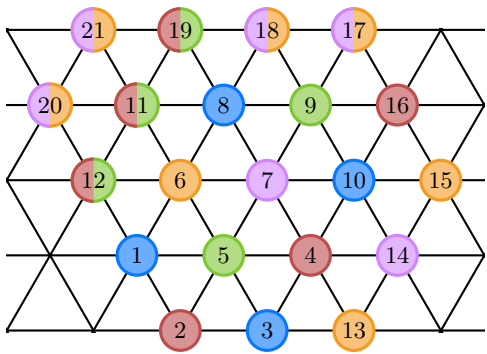

\textbf{Case 2.2a: 10 is blue.}
(13-16) One of 13 and 14 is purple; otherwise we use a purple/red swap at 7,4,
and recolor 6
purple, which gives purple/blue parallel pairs centered at 5.  Similarly, one of
13 and 14 is orange; otherwise, a purple/orange swap at 6,7 reduces to the previous
sentence.  So 13 and 14 are purple and orange in some order.  We assume 13 is
orange and 14 is purple; otherwise a purple/orange swap at 6,7 reduces to this
case.  If 15 is green, then we have green/purple parallel pairs centered at 10.
And if 15 is red, then we have the good template $\{1,2,3,4,8,15\}$; the prefix
is 5, 6, 7, 10 and the bonus vertex is the neighbor of 3 that forms a triple
with 2 and 4.
So 15 is orange. Now 16 is red; otherwise we have a purple triple~centered~at~10.  

(17-19) Since we have no blue triple, 17 is purple or orange.
If 18 is red, then we have red/blue parallel pairs centered at 9.
Thus, 18 is purple or orange.
If 19 is purple or orange, then we either have orange/purple parallel pairs
centered at 8, or we get them after an orange/purple swap at 6,7.
Thus, 19 is red or green.

(20-21) We show that 20 and 21 are both purple or orange.  Suppose 11 is red.
If neither 20 nor 21 is orange, then a red/orange swap at 11,6 gives a red
triple centered at 5.  If neither 20 nor 21 is purple, then a purple/orange
swap at 6,7 followed by a red/purple swap at 11,6 again gives a red triple
centered at 5.  Instead, assume 11 is green.  If neither 20 nor 21 is
orange, then an orange/green swap at 11,6,5 gives orange/purple parallel pairs
centered at 4.  And if neither 20 nor 21 is purple, then we use a purple/orange
swap at 6,7, followed by a green/purple swap at 11,6,5, followed by a
purple/orange swap at 5,7; this again gives orange/purple parallel pairs
centered at 4.  So 20 and 21 are orange and purple.  If 11 is green, then a
green/blue swap at 11,8,9,10 gives a green triple centered at 7.  So assume 11
is red, and use a red/blue swap at 11,8.  If 18 is purple, then a red/orange
swap at 8,6 gives a red triple centered at 5; so assume 18 is orange.  Now a
purple/orange swap at 6,7 and a purple/red swap at 8,6 again gives a
red triple~centered~at~5.

\textbf{Case 2.2b: 10 is orange.}
(13-14) Note that 13 or 18 is purple; otherwise, we use a purple/red swap at 7,4 and
recolor 6 purple, which gives purple/blue parallel pairs centered at 5.  Suppose
that 18 is purple.  Now 13 is not green; otherwise we have green/purple parallel
pairs centered at 4.  So 13 is orange.  Now we reduce to Case 1 with 13, 18, 10,
7, in place of 1, 2, 3, 4 (and orange, purple, red in place of blue, red,
green).  Thus, 13 is purple.  
If 19 is red, then we have the good template
$\{1,2,3,4,8,19\}$; the prefix is 5, 6, 7, 10 and the bonus vertex is the
neighbor of 3 that forms a triple with 2 and 4.  So 19 is not red.
Thus, 14 is red;
otherwise a red/orange swap at 4,10
gives orange/blue parallel pairs centered at 5.  

\begin{wrapfigure}{l}{0.45\textwidth}
\centering
\begin{tikzpicture}[scale=.95]
\begin{scope}
\clip(-1.73,-1.02) rectangle (4.62,3.02);  
\foreach \ang in {0, 60, 120}  
{
\begin{scope}[rotate=\ang]
\foreach \i in {-5,...,5} 
\draw[thick] (-5.5, \i) -- (5.5,\i);
\end{scope}
}
\end{scope}
\tikzstyle{uStyle}=[shape = circle, minimum size = 15pt, inner sep =
1.5pt, outer sep = 0pt, draw, fill=white]

\begin{scope}[xscale=.577] 
\foreach \i/\j/\lab/\col in {
0/0/1/blue, 
1/-1/2/red, 
3/-1/3/blue, 4/0/4/red,
2/0/5/green, 
1/1/6/orange,
3/1/7/purple, 
2/2/8/blue, 
4/2/9/green, 
0/2/11/green,
-1/1/12/red,
5/-1/13/purple,
6/2/14/red,
5/3/15/purple, 
3/3/17/red,
1/3/16/purple,
5/1/10/orange} 
\draw[very thick] (\i,\j) node[uStyle, draw=my\col, fill=my\col!60!white]
{\small{\lab}};  

\foreach \i/\j/\lab/\colA/\colB in {
0/2/11/green/red, 
-1/1/12/green/red,
7/1/19/green/blue, 
6/0/18/green/blue}
{
    \draw[very thick, my\colA, fill=my\colA!60!white] (\i,\j) circle (14pt and 8pt); 
    \fill[my\colB!60!white] (\i,\j) --+ (90:8pt) arc (90:270:14pt and 8pt) -- cycle; 
    \draw[very thick, my\colB] (\i,\j) ++ (90:8pt) arc (90:270:14pt and 8pt); 
    \draw (\i,\j) node[uStyle, draw=none, fill=none] {\small{\lab}};
}

\end{scope}
\end{tikzpicture}
\caption{\label{crossing-pairs-fig5}%
Case 2.2b: 10 is orange.} 
\end{wrapfigure}
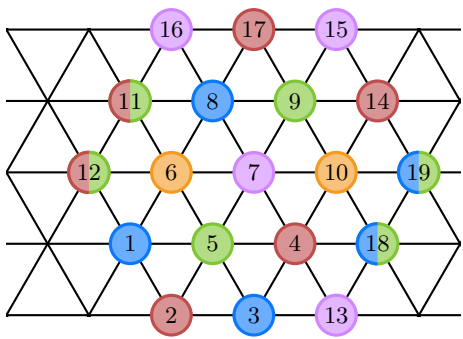

(15-17) If 15 is orange, then we have the good template $\{1,2,3,4,6,10,15\}$, with
prefix 5, 7, 9.  Our bonus vertex is 13, which allows us to finish by
Lemma~\ref{helper-lem}.
Suppose 15 is blue.  If 19 is not purple, then a purple/orange swap at 6,7,10
and a red/orange swap at 4,7 gives red/blue parallel pairs centered at 9.
So assume 19 is purple.  Now 18 is blue; otherwise we recolor 10 blue and get a
blue triple centered at 9.  Now 17 is orange; otherwise an orange/green swap at
10,9 gives parallel orange/green pairs centered at 7.  Finally, a green/purple
swap at 5,7,9 and a red/green swap at 4,7 gives red/blue parallel pairs centered
at 9.  Thus, 15 is not blue.
So 15 is purple, since it sees red and green and is not orange or blue.  Thus,
16 is purple; otherwise a purple/blue swap at 7,8 gives a blue triple centered
at 5.
If 11 is red and 12 is green,
then a green/blue swap at 8,9 gives a green triple centered at 6.  So 11 is
green and 12 is red.
If 17 is orange, then we recolor 8 red and recolor 7 blue, which gives a blue
triple centered at 5.  So 17 is red.

(18-19) Now 18 or 19 is green; otherwise a green/orange swap at 9,10 gives
green/orange parallel pairs centered at 7.  And 18 or 19 is blue; otherwise 
we recolor 10 blue, which gives blue/red parallel pairs
centered at 9.  So 18 and 19 are blue and green (in some order).  Now we use a
purple/orange swap at 6, 7, 10 followed by a red orange swap at 4, 7.  This
gives a red triple centered at 9.
\hfill$\qed$
\smallskip

Let a \Emph{pair} denote the template, with one color and two vertices, induced
by vertices $\{1,3\}$ in Figure~\ref{crossing-pairs-fig5}.

\begin{lem}
Every 5-coloring of $G$ is equivalent to a 5-coloring with a pair.
\label{get-a-pair-lem}
\end{lem}

\noindent
\textit{Proof.}~Six vertices in a coloring form \Emph{alternating sets} if they are
colored as vertices
1,2,3,4,5,6\ in Figure~\ref{get-a-pair-fig2} (not Figure~\ref{get-a-pair-fig1}),
ignoring all other colors in that picture.
We first show that each 5-coloring is equivalent to a 5-coloring that 
contains either a pair or alternating sets.  Assume instead that $\vph$ is
a counterexample to this statement.  

\begin{wrapfigure}{l}{0.45\textwidth}
\centering
\begin{tikzpicture}[scale=.95]
\begin{scope}
\clip(-2.88,-0.02) rectangle (2.88,3.02);  
\foreach \ang in {0, 60, 120}  
{
\begin{scope}[rotate=\ang]
\foreach \i in {-5,...,5} 
\draw[thick] (-5.5, \i) -- (5.5,\i);
\end{scope}
}
\end{scope}
\tikzstyle{uStyle}=[shape = circle, minimum size = 15pt, inner sep =
1.5pt, outer sep = 0pt, draw, fill=white]

\begin{scope}[xscale=.577] 
\foreach \i/\j/\lab/\col in {
0/0/1/red, 0/2/4/orange, 2/0/3/orange, 2/2/2/red,
4/0/5/red,
-2/2/6/red, 1/1/7/green, 3/1/8/blue, -2/0/9/blue,
-1/1/10/purple, 4/2/11/purple, 1/3/12/blue, -1/3/13/green, -3/1/14/green} 
\draw[very thick] (\i,\j) node[uStyle, draw=my\col, fill=my\col!60!white]
{\small{\lab}};  
\end{scope}
\end{tikzpicture}
\caption{\label{get-a-pair-fig1}%
A claim in the proof of Lemma~\ref{get-a-pair-lem}.}
\end{wrapfigure}
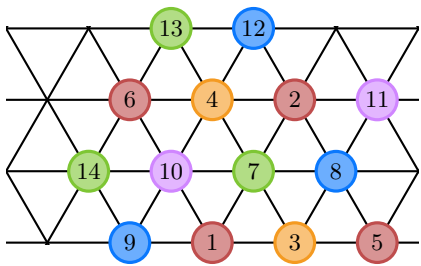

(1-6) Since we have no pair, we also have no
triple.  By Pigeonhole, two colors appear exactly twice on the neighborhood of
7.  Since we have no pair, these repeated colors appear ``across'' from
each other; that is, each color appears on two vertices with 7 as their
unique common neighbor.  By rotational symmetry, we assume 1 and 2
are red, and 3 and 4 are orange, as in Figure~\ref{get-a-pair-fig1}.
If no neighbor of 3 is red, and no neighbor of 1 is orange, then an orange/red
swap at 1,3 gives orange on both 1 and 4, a pair.  Thus, 
either 5 is red or 9 is orange; by symmetry, we assume that 5 is red.  By
repeating the same argument for 2 and 4, we assume that 6 is red; if instead 11
is orange, then we have alternating sets, as claimed.

(7-14) Now 7 and 8 are new colors, green and blue, respectively.  If 9 is
orange, then we have alternating
sets; and if 9 is green, then we have a green pair, 9 and 7.  If 9 is purple, then
we recolor 7 with purple and get a purple pair.  So 9 is blue.  This implies that
10  is purple.  Clearly, 11 is not red or blue.  It is also not green, since
$\vph$ has no pair, and it is not orange, since $\vph$ does not have alternating
sets.  So 11 is purple.
If 12 is green or purple, then $\vph$ contains a pair; so 12 is blue.  Now 13
and 14 are both green; for each vertex, three colors appear on its neighbors and a
fourth is forbidden, since $\vph$ contains no pair.  Now we use a purple/orange
swap at 10,4.  This gives an orange pair at 10 and 3.  This proves the claim
that $\vph$ is equivalent to a 5-coloring that contains either a pair or
alternating sets.

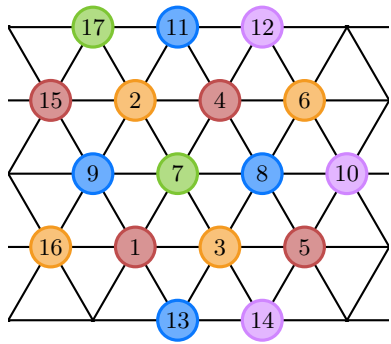
\begin{wrapfigure}{l}{0.45\textwidth}
\centering
\begin{tikzpicture}[scale=.975]
\begin{scope}
\clip(-2.31,-2.02) rectangle (2.88,2.02);  
\foreach \ang in {0, 60, 120}  
{
\begin{scope}[rotate=\ang]
\foreach \i in {-5,...,5} 
\draw[thick] (-5.5, \i) -- (5.5,\i);
\end{scope}
}
\end{scope}
\tikzstyle{uStyle}=[shape = circle, minimum size = 15pt, inner sep =
1.5pt, outer sep = 0pt, draw, fill=white]

\begin{scope}[xscale=.577] 
\foreach \i/\j/\lab/\col in {
-1/-1/1/red, -1/1/2/orange, 1/-1/3/orange, 1/1/4/red,
3/-1/5/red, 3/1/6/orange, 0/0/7/green, 2/0/8/blue,
-2/0/9/blue, 4/0/10/purple, 0/2/11/blue, 2/2/12/purple,
0/-2/13/blue, 2/-2/14/purple, -3/1/15/red, -3/-1/16/orange,
-2/2/17/green}
\draw[very thick] (\i,\j) node[uStyle, draw=my\col, fill=my\col!60] {\small{\lab}};  
\end{scope}
\end{tikzpicture}
\caption{Finishing the proof of Lemma~\ref{get-a-pair-lem}.\label{get-a-pair-fig2}}
\end{wrapfigure}

(1-10) Now we assume that $\vph$ contains alternating sets, with 1, 4, 5 red and
with 2, 3, 6 orange, as in  Figure~\ref{get-a-pair-fig2}.
We also assume 7 is green and 8 is blue. 
(a) First suppose that 9 and 10 are both purple.  At least one of 11 and 12 is blue
or green.  So either $\vph$ contains as a pair one of $\{7,11\}$ or $\{8,12\}$ or the
resulting coloring contains one of these pairs after a green/blue swap at 7,8.
So assume that 9 and 10 are not both purple. (b) Next suppose that 9 is blue and
10 is green.  If none of 11, 12, 13, 14 is purple, then we recolor both 3 and 4
purple, which gives a pair.  If instead one of 11, 12, 13, 14 is purple,
then recoloring 7 or 8 purple gives a purple pair.  Thus, we assume that 
exactly one of 9 and 10 is purple.  (c) By reflectional symmetry, we assume
that 9 is blue and 10 is purple.

(11-18) Since 7 is green, 11 is not green.  Similarly, 11 is not purple, since
then we can recolor 7 purple.  So 11 is blue.  If 12 is green, then we can recolor 7
purple and recolor 8 green.  So 12 is purple.  By reflectional symmetry, 13 is
blue and also 14 is purple.  Note that 15 is not green, since 7 is green.  And
15 is not purple, since then we could recolor 7 purple to get a purple pair.  So
15 is red.  Similarly, 16 is orange.  Next, 17 is green; otherwise, an
orange/green swap at 2,7,3 gives the orange pair $\{6,7\}$.  
Now recoloring 2 purple gives a purple pair at $\{2,12\}$.
\hfill $\qed$

\begin{lem}
If $\vph$ is a 5-coloring of $G$ with a pair, then $\vph$ is equivalent to a
coloring that contains either a triple, or parallel pairs, or crossing pairs.
In particular, $\vph$ is nice.
\label{pair-lem}
\end{lem}
\textit{Proof.}
The second statement follows immediately from the first, combined with
Lemmas~\ref{triple-lem}, \ref{parallel-pairs-lem}, and~\ref{crossing-pairs-lem}.
So we now prove the first; 
instead assume that $\vph$ is a counterexample.

(1-8) By assumption, 1 and 2 have a common color, red; see
Figure~\ref{case1pair}.
Now 3 and 4 have new colors: purple and green, respectively.  Further, 5 is also
purple. Suppose instead that 5 is orange; now each of 7 and 8 is either
purple or blue.  If either is purple, then we have red/purple crossing pairs
centered at 4.  Otherwise, both 7 and 8 are blue, so we have red/blue parallel
pairs centered at 4.  Thus, we assume that 5 is purple; similarly, we assume that 6
is green.  Below we frequently use this argument implicitly.
Finally, 7 and 8 have new colors; if they are not distinct, then we have
parallel pairs centered at 4.  So assume 7 is orange and 8 is blue.
Now we consider 4 cases, depending on the color of 9.

\begin{wrapfigure}{L}{0.45\textwidth}
\centering
\begin{tikzpicture}[scale=.885]
\begin{scope}
\clip(-2.88,-2.02) rectangle (2.88,2.02);  
\foreach \ang in {0, 60, 120}  
{
\begin{scope}[rotate=\ang]
\foreach \i in {-5,...,5} 
\draw[thick] (-5.5, \i) -- (5.5,\i);
\end{scope}
}
\end{scope}
\tikzstyle{uStyle}=[shape = circle, minimum size = 15pt, inner sep =
1.5pt, outer sep = 0pt, draw, fill=white]

\begin{scope}[xscale=.577] 
\foreach \i/\j/\lab/\col in {
0/0/1/red, 0/-2/2/red, -1/-1/3/purple, 1/-1/4/green,
3/-1/5/purple, -3/-1/6/green, 2/0/7/orange, 2/-2/8/blue,
-1/1/9/orange, 2/2/10/red, 1/1/11/purple, -2/0/12/blue,
-2/-2/13/orange, -2/2/16/red, -3/1/17/purple, -4/0/18/red,
0/2/19/blue} 
\draw[very thick] (\i,\j) node[uStyle, draw=my\col, fill=my\col!60!white]
{\small{\lab}};  

\foreach \i/\j/\lab/\colA/\colB in {3/1/14/blue/green, 4/0/15/blue/green}
{
    \draw[very thick, my\colA, fill=my\colA!60!white] (\i,\j) circle (14pt and 8pt); 
    \fill[my\colB!60!white] (\i,\j) --+ (90:8pt) arc (90:270:14pt and 8pt) -- cycle; 
    \draw[very thick, my\colB] (\i,\j) ++ (90:8pt) arc (90:270:14pt and 8pt); 
    \draw (\i,\j) node[uStyle, draw=none, fill=none] {\small{\lab}};
}

\end{scope}
\end{tikzpicture}
\caption{Case 1: 9 is orange. 
\label{case1pair}}
\end{wrapfigure}
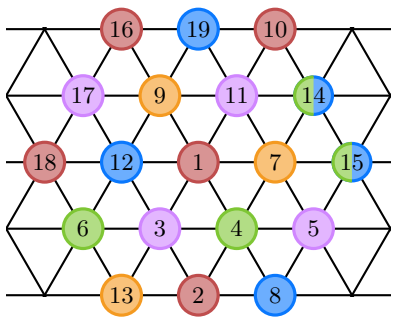

\textbf{Case 1: 9 is orange.}
(10-13) Since 7 and 9 are both orange, 10 is red, like 1, and 11 is purple, like 3.
Now 12 is blue.  And 13 is orange; otherwise, a purple/orange swap at 3 gives an
orange triple centered at 1.

(14-16) One of 14 and 15 is blue; otherwise a blue/orange swap at 7 gives
blue/red parallel pairs centered at 4.  And one of of 14 and 15 is green;
otherwise, an orange/green
swap at 7,4 gives orange/red crossing pairs centered at 3.  
Note that 17 is not red, since 1 is red and 6 and 9 have distinct colors.
So 16 is red; otherwise a red/orange swap at 9,1,7 gives a red triple centered at 11.

(17-19) Note that 18 cannot be purple, since 3 is purple, but 6 and 9 have
distinct colors.
So 17 is purple; otherwise, a blue/purple swap at 12,3 gives blue/red crossing
pairs centered at 4.  Further, 18 cannot be orange, since then 
a red/blue swap at 12,1 gives red/orange crossing pairs centered at 17.
So 18 is red.
If 19 is green, then an orange/blue swap at 9,12 gives orange/red parallel pairs
centered at 3.  So 19 is blue. But now we recolor 9 with green, and recolor
12 with orange.  This also gives orange/red parallel pairs centered at 3.

\begin{wrapfigure}{L}{0.45\textwidth}
\centering
\begin{tikzpicture}[scale=.885]
\begin{scope}
\clip(-2.88,-3.015) rectangle (2.88,2.015);
\foreach \ang in {0, 60, 120}
{
\begin{scope}[rotate=\ang]
\foreach \i in {-5,...,5}
\draw[thick] (-5.5, \i) -- (5.5,\i);
\end{scope}
}
\end{scope}
\tikzstyle{uStyle}=[shape = circle, minimum size = 15pt, inner sep =
1.5pt, outer sep = 0pt, draw, fill=white]

\begin{scope}[xscale=.577] 
\foreach \i/\j/\lab/\col in {
0/0/1/red, 0/-2/2/red, -1/-1/3/purple, 1/-1/4/green,
3/-1/5/purple, -3/-1/6/green, 2/0/7/orange, 2/-2/8/blue,
-1/1/9/purple, 1/1/10/blue, -2/0/11/orange, -4/0/12/red,
-3/1/13/blue, -2/-2/14/blue, 3/1/15/green, 2/2/16/red,
-2/2/17/red, 0/2/18/orange} 
\draw[very thick] (\i,\j) node[uStyle, draw=my\col, fill=my\col!60!white] {\small{\lab}};
\foreach \i/\j/\lab/\col in 
{4/0/21/white} 
\draw[very thick] (\i,\j) node[uStyle, draw=black!60!white, fill=white] {\small{\lab}};

\foreach \i/\j/\lab/\colA/\colB in {-1/-3/19/orange/green, 1/-3/20/orange/green}
{
    \draw[very thick, my\colA, fill=my\colA!60!white] (\i,\j) circle (14pt and 8pt); 
    \fill[my\colB!60!white] (\i,\j) --+ (90:8pt) arc (90:270:14pt and 8pt) -- cycle; 
    \draw[very thick, my\colB] (\i,\j) ++ (90:8pt) arc (90:270:14pt and 8pt); 
    \draw (\i,\j) node[uStyle, draw=none, fill=none] {\small{\lab}};
}

\end{scope}
\end{tikzpicture}
\caption{Case 2: 9 is purple.\label{case2pair}}
\end{wrapfigure}
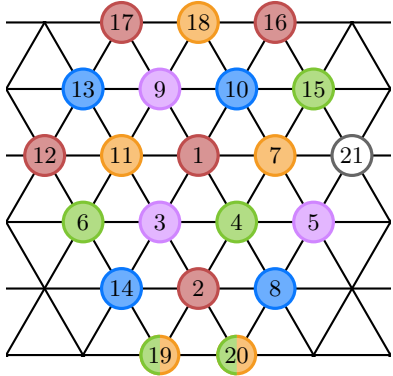

\textbf{Case 2: 9 is purple.}
(10-14) If 10 is green, then we have green/purple parallel pairs centered at 1,
so 10 is blue.
Since 3 and 9 are purple, 11 is orange and 12 is red.  If 13 is green, then we
have green/purple parallel pairs centered at 11; so 13 is blue.  If 14 is orange,
then we have orange/red parallel pairs centered at 3; so~14~is~blue.  

(15-16) Note that 21 is not green, since 4 is green but 5 and 10 are distinct
colors.  Now 15 is green; otherwise, a green/orange swap at 4,7 gives
orange/purple crossing pairs centered at 1.  Note that 18 is not red, since 1
is red but 9 and 15 are distinct colors.  So 16 is red; otherwise a red/blue
swap at 1,10 gives blue/purple crossing pairs centered at 11.

(17) Suppose 17 is not red.  If neither 19 nor 20 is purple, then a 
red/purple swap at 2,3,1,9  gives a red triple centered at 11.  If neither 19 nor
20 is green, then a green/red swap at 2,4,1 gives a green triple centered at 3.
So 19 and 20 are purple and green.  Now we recolor 2 with orange and do a
green/red swap at 4,1, which gives green/purple crossing pairs centered at 11. 
So 
17 is red.

(18-20) Now 18 is orange; otherwise an orange/purple swap at 9,11,3 gives an orange
triple centered at 1.  If neither 19 nor 20 is green, then a green/red swap at
2,4,1 gives a green triple centered at 3.  If neither 19 nor 20 is orange, then
we recolor 2 with orange and use a green/red swap at 4,1.  This gives green/purple
crossing pairs centered at 11.  So 19 and 20 are green and orange.
Now we recolor 9 with green.  Afterwards, a red/purple swap at 1,3,2
gives a purple triple centered at 4. 

\textbf{Case 3: 9 is blue.}
(10-11) Now 10 is orange, since all other colors are used on its neighborhood. 
If 11 is orange, then we have orange/red parallel pairs centered at 3.
So 11 is blue.

(12-13) If 12 is green, then we are actually in Case 2, by reflecting
horizontally and interchanging colors green and purple.  So we assume that 12
is not green; thus, 12 is purple.  Similarly, 17 is not green, 
by reflecting both vertically and horizontally, and interchanging both
orange/blue and green/purple.  Now 13 is green;
otherwise a green/red swap at 1,4,2 gives a green triple centered at 3.

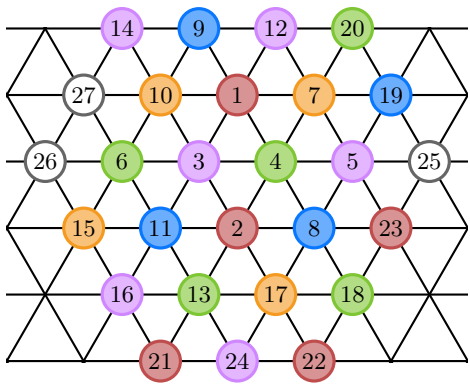
\begin{wrapfigure}{L}{0.45\textwidth}
\centering
\begin{tikzpicture}[scale=.885]
\begin{scope}
\clip(-3.46,-2.015) rectangle (3.46,3.015);
\foreach \ang in {0, 60, 120}
{
\begin{scope}[rotate=\ang]
\foreach \i in {-5,...,5}
\draw[thick] (-5.5, \i) -- (5.5,\i);
\end{scope}
}
\end{scope}
\tikzstyle{uStyle}=[shape = circle, minimum size = 15pt, inner sep =
1.5pt, outer sep = 0pt, draw, fill=white]

\begin{scope}[xscale=.577] 
\foreach \i/\j/\lab/\col in {
0/2/1/red, 0/0/2/red, -1/1/3/purple, 1/1/4/green,
3/1/5/purple, -3/1/6/green, 2/2/7/orange, 2/0/8/blue,
-1/3/9/blue, -2/2/10/orange, -2/0/11/blue, 1/3/12/purple,
-1/-1/13/green, -3/3/14/purple, -4/0/15/orange, -3/-1/16/purple,
1/-1/17/orange, 3/-1/18/green, 4/2/19/blue, 3/3/20/green,
-2/-2/21/red, 2/-2/22/red, 4/0/23/red, 0/-2/24/purple}
\draw[very thick] (\i,\j) node[uStyle, draw=my\col, fill=my\col!60!white] {\small{\lab}};
\foreach \i/\j/\lab/\col in {5/1/25/white, -5/1/26/white, -4/2/27/white}
\draw[very thick] (\i,\j) node[uStyle, draw=black!60!white, fill=white] {\small{\lab}};
\end{scope}
\end{tikzpicture}
\caption{Case 3 in the proof of Lemma~\ref{pair-lem}.\label{case3pair}}
\end{wrapfigure}

(14) Note that 27 is not purple, since 3 is purple but 6 and 9 have distinct
colors.  So, if 14 is not purple, then a purple/orange swap at 10,3 gives
purple/orange parallel pairs, centered at 1.
Thus, 14 is purple.  

(15-16) Now 15 and 16 must be orange and
purple.  If neither is orange, then recoloring 11 with orange gives orange/red
parallel pairs centered at 3.  If neither is purple, then a purple/blue swap at 11,3
gives blue/red crossing pairs centered at 4.  So 15 and 16 are orange and
purple.  Suppose that 15 is purple.  If 26 is blue, then we have blue/purple
crossing pairs centered at 6.  And if 26 is orange, then we have orange/purple
parallel pairs centered at 6.  So 26 is red.  This means that 27 is blue.
Now a green/orange swap at 6,10 gives green/red crossing pairs centered at 3.
So we assume 15 is not purple.  Thus, 15 is orange and 16 is purple.

(17-18) If 17 is purple, then we recolor 2 with orange, recolor 11 with red, and
use a red/green swap at 1,4.  This gives red/green parallel pairs centered at 3. 
Thus, 17 is orange.  Now 23 is not green, since 4 is green but 5 and 17
have distinct colors.  So 18 is green; otherwise, a green/blue swap at 8,4
gives green/blue parallel pairs centered at 2.

(19-20) If neither 19 nor 20 is blue, then a blue/orange swap at 7 gives
blue/red parallel pairs centered at 4.  If neither 19 nor 20 is green, then
a green/orange swap at 7,4 gives orange/red crossing pairs centered at 3.  So
19 and 20 are blue and green.
If 19 is green, then 23 is orange since 7 is orange.  Thus, 25 is red;
otherwise, we can recolor 5 red and get a red triple centered at 4.  Now a
blue/red swap at 8,2,11 gives red/green parallel pairs centered at 5. 
Thus, 19 is blue and 20 is green.

(21-24) Note that 24 cannot be red, since 2 is red and 16 and 17 have distinct
colors.  So 21 is red; otherwise, a red/green swap at 13,2,4,1 gives a green
triple centered at 3.  Similarly, 22 is red; otherwise, a red/orange swap at
2,17 and a red/green swap at 1,4 gives a red triple centered at 8.
Now 23 must be red; otherwise a red/blue swap at 8,2,11 gives 11 and 21 red but
16 and 17 with distinct colors, a contradiction.  Finally, 24 is purple.  If
not, then we recolor 17 purple and repeat the argument above showing that 17 is
orange, getting a contradiction.  Now an orange/blue swap at 17,8 gives
orange/red~parallel~pairs~centered~at~4. 

\textbf{Case 4: 9 is green.} 
(9-11) 
We assume that 10 is green; 
otherwise reflecting vertically (interchanging orange/blue) reduces to an
earlier case.  Further, either 11 or 24 is purple.  Otherwise, a purple/red
swap at 3 gives a purple triple centered at 4.  By possibly reflecting
vertically, we assume 11 is purple.

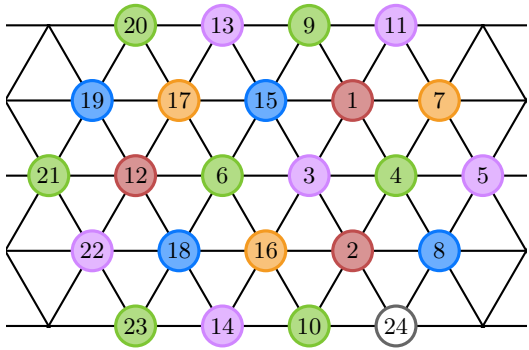
\begin{wrapfigure}{L}{0.45\textwidth}
\centering
\begin{tikzpicture}
\begin{scope}
\clip(-4.6,-1.015) rectangle (2.31,3.015);
\foreach \ang in {0, 60, 120}
{
\begin{scope}[rotate=\ang]
\foreach \i in {-5,...,5}
\draw[thick] (-5.5, \i) -- (5.5,\i);
\end{scope}
}
\end{scope}
\tikzstyle{uStyle}=[shape = circle, minimum size = 15pt, inner sep =
1.5pt, outer sep = 0pt, draw, fill=white]

\begin{scope}[xscale=.577] 
\foreach \i/\j/\lab/\col in {
0/2/1/red, 0/0/2/red, -1/1/3/purple, 1/1/4/green,
3/1/5/purple, -3/1/6/green, 2/2/7/orange, 2/0/8/blue,
-1/3/9/green, -1/-1/10/green, 1/3/11/purple, 
-5/1/12/red, 
-3/3/13/purple,
-3/-1/14/purple, 
-2/2/15/blue, 
-2/0/16/orange, 
-4/2/17/orange, 
-4/0/18/blue, 
-6/2/19/blue,
-5/3/20/green,
-7/1/21/green,
-6/0/22/purple,
-5/-1/23/green}
\draw[very thick] (\i,\j) node[uStyle, draw=my\col, fill=my\col!60!white] {\small{\lab}};
\foreach \i/\j/\lab/\col in {1/-1/24/orange}
\draw[very thick] (\i,\j) node[uStyle, draw=black!60!white, fill=white] {\small{\lab}};
\end{scope}
\end{tikzpicture}
\caption{Case 4 in the proof of Lemma~\ref{pair-lem}.\label{case4pair}}
\end{wrapfigure}

(12) Note that 15 and 16 must each be either orange or blue, and they must
have distinct colors; otherwise, we have parallel pairs centered at 3.  
Denote the colors of 15 and 16 by $\alpha$ and $\beta$, respectively.
If neither 13 nor 17 is purple, then an $\alpha$/purple swap at 15,3 gives
$\alpha$/red crossing pairs centered at 4.  And if neither 13 nor 17 is $\beta$,
then an $\alpha$/$\beta$ swap at 15 gives $\beta$/red parallel pairs centered at
3.  So 13 and 17 are $\beta$ and purple (in some order).  By a similar
argument, 18 and 14 are $\alpha$ and purple (in some order).  Thus, 12 is red;
otherwise, we recolor 6 red, which gives a red triple centered at 3.

(13-14) If 17 is purple, then either 13 is orange and 15 is blue or vice versa.  In the
first case, a blue/red swap at 15,1 gives red/purple crossing
pairs centered at 6.  In the second case, we recolor 1 blue, which gives
blue/purple parallel pairs centered at 15.  So 13 is purple.
Suppose 18 is purple.  If 15 is blue, then a red/blue swap at 15,1 gives
red/purple parallel pairs centered at 6.  If 15 is orange, then we recolor 1
blue and recolor 15 red.  Again, we get red/purple parallel pairs centered at 6.
So 14 is purple.

(15-18) Suppose 15 is orange and 16 is blue.  If 24 is purple, then recolor 1 blue, 2
orange, 15 red, and 16 red; this gives a red triple centered at 6.  So 24 is
orange.  Recolor 1 blue.  Use a red/purple swap at 2,3, and a red/orange
swap at 3,15.  This gives orange/purple parallel pairs centered at 16.
So 15 is blue and 16 is orange.
This implies that 17 is orange and 18 is blue.

(19-23) Note that 19 cannot be green, since 6 is green, but 12 and 13 have distinct
colors.  If 19 is purple, then a red/blue swap at 15,1 gives red/purple parallel
pairs centered at 17.  So 19 is blue.  If 20 is red, then a red/blue swap at
1,15 gives a red triple centered at 17.  So 20 is green.  Note that 21 or 22
must be green; otherwise, a red/green swap at 12,6 gives a red triple centered
at 3.  And 22 cannot be green, since 6 is green but 12 and 14 have distinct
colors.  So 21 is green.  If 22 is orange, then we recolor 12 purple and recolor
6 red, which gives a red triple centered at 3.  So 22 is purple.  Finally, 23
must be green; otherwise we have purple/red or purple/orange crossing pairs
centered at 18.  Now a blue/orange swap at 18,16 gives blue/orange parallel
pairs centered at 6.
\hfill $\qed$
\bigskip

Now we prove the main result of this section.

\begin{lem}
If $G$ is a 6-regular toroidal graph with edge-width at least 7, then every
5-coloring of $G$ is 5-equivalent to a 5-coloring that contains a good
4-template.  
\label{main-lem}
\end{lem}
\begin{proof}
Let $G$ be a graph satisfying the hypothesis, and let $\vph_0$ be a 5-coloring of
$G$.  By Lemma~\ref{get-a-pair-lem}, $\vph_0$ is 5-equivalent to some 5-coloring
$\vph_1$ that contains a pair.  By Lemma~\ref{pair-lem}, $\vph_1$ is
5-equivalent to some 5-coloring $\vph_2$ that contains either a triple, or
parallel pairs, or crossing pairs.  By Lemmas~\ref{triple-lem},
\ref{parallel-pairs-lem}, and~\ref{crossing-pairs-lem}.  $\vph_2$ is
5-equivalent to some 5-coloring $\vph_3$ the contains a triple, and thus
contains a good 4-template.
\end{proof}

\section{Extensions and Open Questions}
\label{ext-sec}
We begin the section by sketching the proof of Theorem~\ref{main-cor}.
For reference, we repeat the statement below.

\begin{thmcor}
If $G$ is a triangulated toroidal grid $T[a\times b]$ with $a\ge 6$ and $b\ge
6$, then all 5-colorings of $G$ are 5-equivalent.
\end{thmcor}
\begin{proof}[Proof Sketch.]
We sketch how to modify the proof of the Main Theorem to prove
Theorem~\ref{main-cor}.
Recall that our Main Theorem handles the case that $a\ge 7$ and $b\ge 7$, since
then $G$ has edge-width at least 7.  By symmetry, we can assume that $a\le b$.
So assume that $a=6$.  By Lemma~\ref{4colorable-lem}, we know $G$ is
4-colorable; so Lemma~\ref{switch-lem} still applies.  In fact, the only place
that our proof uses edge-width 7 is when we apply Lemma~\ref{loc-help-lem} to
show that a subgraph $H$ is locally connected.  However, in the proof of
Lemma~\ref{loc-help-lem} we only need the fact that $H$ does not contain five
vertices on a non-contractible cycle of length 6.
The graph $T[6\times 6]$ has exactly 18 non-contractible cycles of length 6.
These run along the 6 rows, 6 columns, and 6 diagonals.  And for $T[6\times b]$
with $b\ge 7$, we have $b$ non-contractible cycles of length 6; each runs along
a column.  So, to adapt
the proof to our present situation, it suffices to check that $H$ does not
contain exactly 5 vertices on any of these non-contractible 6-cycles.  In the
proofs of Lemmas~\ref{triple-lem} and~\ref{helper-lem}, this is true.  However,
in the proofs of Lemmas~\ref{parallel-pairs-lem}
and~\ref{crossing-pairs-lem}, there is sometimes a single non-contractible cycle $C$
of length 6 with 5 of its vertices in $H$.  In this case, we simply add
the final vertex $v$ of $C$ to $T$ and to $H$, making $v$ the sole vertex in its
color in $T$.  It is straightforward to check that the subgraph $G-H-\{v\}$ is still
connected.  Thus, the proof of Lemma~\ref{loc-help-lem} still applies.
\end{proof}

In the introduction we mentioned that Mohar proved that if $G$ is a planar
graph with $\chi(G)=k$, then all $(k+1)$-colorings of $G$ are $(k+1)$-equivalent.  
Mohar also constructed, for any positive integers $k,\ell$ with $k<\ell$ a
$k$-chromatic graph with a single $k$-coloring and with two $\ell$-colorings
that are not $\ell$-equivalent (the graph is simply the categorical product of
$K_k$ and $K_{\ell}$; the details are in Proposition~2.1 of~\cite{mohar}).  We
would still like to find larger classes of graphs
$G$ for which all $(\chi(G)+1)$-colorings are $(\chi(G)+1)$-equivalent.  Our
conjectures below focus on graphs embedded in the torus, and other surfaces.
We first state a lemma that will likely be useful in studying this problem.
We call a template $T$ in a graph $G$ \emph{$k$-good} if $G_T$ is
$k$-degenerate.

\begin{lem}
\label{switch-gen-lem}
Let $G$ be a graph with $\chi(G)\le k$.  If $\vph_1$ and $\vph_2$ are
$(k+1)$-colorings of $G$ that contain monochromatic $k$-good templates $T_1$
and $T_2$, respectively, then $\vph_1$ and $\vph_2$ are $(k+1)$-equivalent.
\end{lem}

We omit the proof, since it is nearly identical to that of
Lemma~\ref{switch-lem}.  

\begin{conj}
If $G$ is a triangulated toroidal grid $T(a\times b)$ (say with $a\ge 3$ and
$b\ge 3$), then all 5-colorings of $G$ are 5-equivalent.
\end{conj}

\begin{conj}
If $G$ is a 4-chromatic toroidal graph, then all 5-colorings of $G$ are
5-equivalent.
\end{conj}

\begin{conj}
If $G$ is a toroidal graph, then all 5-colorings of $G$ are 5-equivalent.
\end{conj}

\begin{conj}
For every surface $S$ there exists $c_S$ such that if $G$ embeds in $S$
with edge-width at least $c_S$, then all 5-colorings of $G$ are 5-equivalent.
\end{conj}

\section*{Acknowledgments}
Thanks to Marthe Bonamy for suggesting this problem.
Thanks to Bojan Mohar for helpful comments on the history of the classification
of 6-regular toroidal graphs, as well as for bringing our attention to his
construction that we mentioned in Section~\ref{ext-sec}.

\bibliographystyle{siam}
{\footnotesize{\bibliography{refs}}
\end{document}